\documentclass[11pt]{article}
\usepackage{amsmath}
\usepackage{amssymb}
\usepackage{amsthm}
\usepackage{enumerate}
\usepackage{graphicx,color}
\usepackage{xspace}

\parskip \medskipamount
\newtheorem{theorem}{Theorem}[section]
\newtheorem{lemma}[theorem]{Lemma}
\newtheorem{corollary}[theorem]{Corollary}

\newtheorem{proposition}[theorem]{Proposition}
\newtheorem{remark}[theorem]{Remark}
\newtheorem{definition}[theorem]{Definition}

\def\calC{{\mathcal C}}
\def\calL{{\mathcal L}}

\newcommand{\pr}[1]{\operatorname{\mathbf{P}}\left(#1\right)}

\newcommand{\E}[1]{\operatorname{\mathbf{E}}\left[#1\right]}

\newcommand{\Econd}[2]{\operatorname{\mathbf{E}}\left[#1\;\middle\vert\;#2\right]}

\newcommand{\critical}{\mathrm{c}}

\newcommand{\tv}[2]{\|#1-#2\|_\mathrm{TV}}
\newcommand{\tmix}{T_\mathrm{mix}}

\newcommand{\fresh}{\star}
\newcommand{\regeneration}{\Omega_\mathrm{REG}}

\newcounter{notecounter}

\newcommand{\IGNORE}[1]{}

\newcommand{\ndtorus}{\mathbb{T}^{d,n}}
\newcommand{\bfZ}{\mathbb{Z}^d}
\newcommand{\dntorus}{\mathbb{T}^{d,n}}
\newcommand{\twontorus}{\mathbb{T}^{2,n}}
\newcommand{\onentorus}{\mathbb{T}^{1,n}}
\newcommand{\dist}{\mathrm{dist}}

\newcommand{\tmixRW}{T^{\mathrm{RW}}_{\mathrm{mix}}}
\newcommand{\Pruu}[1]{\operatorname{\mathbf{P}}\left[#1\right]}
\newcommand{\bfP}{\mathbf{P}}

\newtheorem{remarks}[theorem]{Remarks}
\def\calN{{\mathcal N}}
\begin{document}

\title{Random walks on dynamical percolation: mixing times, mean
squared displacement and hitting times}

\author{Yuval Peres\thanks{Microsoft Research, Redmond WA, U.S.A.\ \ Email:
        \hbox{peres@microsoft.com}} \and
        Alexandre Stauffer\thanks{University of Bath, Bath, U.K.\ \ Email:
        \hbox{a.stauffer@bath.ac.uk}}\and
        Jeffrey E. Steif\thanks{Chalmers University of Technology
and Gothenburg University, Gothenburg, Sweden\ \ Email:
        \hbox{steif@chalmers.se}}
}

\maketitle
\thispagestyle{empty}

\begin{abstract}
We study the behavior of random walk on dynamical percolation. In this model,
the edges of a graph $G$ are either open or closed and
refresh their status at rate $\mu$ while at the same time a random walker moves on $G$ at rate 1 
but  only along edges which are open.
On the $d$-dimensional torus with side length $n$, we prove that in the subcritical regime, 
the mixing times for both the full system and the random walker are
$n^2/\mu$ up to constants. We also obtain results concerning mean squared 
displacement and hitting times. Finally, we show that the usual recurrence transience dichotomy
for the lattice $\mathbb{Z}^d$ holds for this model as well.

 \medskip\noindent
 \emph{Keywords and phrases.} Percolation, dynamical percolation, random walk, mixing times.
 \newline
 MSC 2010 \emph{subject classifications.}
 Primary 60K35, 
    60K37, 

  \medskip\noindent

\end{abstract}

\section{Introduction} \label{sec:Introduction}

Random walks on finite graphs and networks have been studied for quite some 
time; see \cite{AldousFill}. Here we study random walks on certain randomly evolving 
graphs. The simplest such evolving graph is given by dynamical percolation. Here one 
has a graph $G=(V,E)$ and parameters $p$ and $\mu$ and one
lets each edge evolve independently where an edge in state 0 (absent, closed) switches to 
state 1 (present, open) at rate $p\mu$ and an edge in state 1 switches to 
state 0 at rate $(1-p)\mu$. We assume $\mu\le 1$. Let $\{\eta_t\}_{t\ge 0}$ denote the 
resulting Markov process on $\{0,1\}^E$ whose stationary distribution is product measure
$\pi_p$. We next perform a  random walk on the evolving graph $\{\eta_t\}_{t\ge 0}$
by having the random walker at rate 1 choose a neighbor (in the original graph) uniformly
at random and move there if (and only if) the connecting edge is open at that time. Letting
$\{X_t\}_{t\ge 0}$ denote the position of the walker at time $t$, we have that
\[
\{M_t\}_{t\ge 0}:= \{(X_t,\eta_t)\}_{t\ge 0}
\]
is a Markov process while $\{X_t\}_{t\ge 0}$ of course is not. One motivation for the model is
that in real networks, the structure of the network itself can evolve over time; however the time
scale at which this might occur is much longer than the time scale for the walker itself. 
This would correspond to the case $\mu \ll 1$ which is indeed the interesting regime for our results.

Our first result shows that the usual
recurrence/transience criterion for ordinary random walk holds for this model as well.

\begin{theorem}\label{thm:rec}
(i). If $G=\mathbb{Z}^d$ with $d$ being 1 or 2, then 
for any $p\in [0,1]$, $\mu>0$ and initial bond configuration $\eta_0$, 
we have that, for any $s_0\geq 0$, $\pr{\bigcup_{s\geq s_0} \{X_s= 0\}}=1$. \\
(ii). If $G=\mathbb{Z}^d$ with $d\ge 3$, then
for any $p\in(0,1]$, $\mu>0$ and initial bond configuration $\eta_0$, we have that
\[
\lim_{t\to\infty} X_t=\infty \,\,\, { \rm a.s.}
\]
\end{theorem}

We note that when $G$ is finite and has constant degree, one can check that
$u\times \pi_p$ is the unique stationary distribution and that the process is reversible;
$u$ here is the uniform distribution.

Next, out main theorem gives the mixing time up to constants for 
$\{M_t\}_{t\ge 0}$ and $\{X_t\}_{t\ge 0}$ on the $d$-dimensional discrete torus
with side length $n$, denoted by $\mathbb{T}^{d,n}$, in the subcritical regime for percolation,
where importantly $\mu$ may depend on $n$.

Let $\tv{m_1}{m_2}$ denote the total variation distance between two probability measures
$m_1$ and $m_2$, $\tmix$ denote the mixing time for a Markov chain and let $p_c(d)$ 
denote the critical value for percolation on $\mathbb{Z}^d$.
See Section \ref{sec:Background} for definitions of all these terms.
Next, starting the walker at the origin and taking the initial bond configuration 
to be distributed according to  $\pi_p$, let
\[
\tmixRW(\epsilon) := \inf\{t\geq 0 \colon\tv{\calL(X_t)}{u} \leq \epsilon\}.
\]
(The superscript {\rm RW} refers to the fact that we are only considering the walker here
rather than the full system.) Below $p_\critical(\mathbb{Z}^{d})$ denotes the critical
value for bond percolation on $\mathbb{Z}^{d}$ and $\theta_d(p)$ denotes the probability that the origin
is in an infinite component when the parameter is $p$; see Section \ref{sec:Background}.

\begin{theorem}\label{thm:sub} 
(i). For any $d\ge 1$ and $p\in (0,p_\critical(\mathbb{Z}^{d}))$, there exists 
$C_1< \infty$ such that, for all $n$ and for all $\mu$, 
considering the full system $\{M_t\}_{t\ge 0}$ on $\mathbb{T}^{d,n}$, we have
\[
\tmix \le \frac{C_1n^2}{\mu}.
\]
(ii).  For any $d\ge 1$, $p\in (0,p_\critical(\mathbb{Z}^{d}))$ and $\epsilon<1$,
there exist $C_2>0$ and $n_0>0$ such that, for all $n\ge n_0$ and for all $\mu$,
considering the system on $\mathbb{T}^{d,n}$, we have
\[
\tmixRW(\epsilon) \ge \frac{C_2n^2}{\mu}.
\]
\end{theorem}

\begin{remarks}{\rm
1. (ii) implies that the upper bound in (i) is also
a lower bound up to constants.\\
2. (i) implies, using (\ref{eq:MixingIterating}) in Section \ref{sec:Background},
that the lower bound in (ii) is also an upper bound up to \\
($\epsilon$ dependent) constants.\\
3. The theorem shows that the ``mixing time'' for
the random walk component of the chain is the same as for the full system. However,
as this component is not Markovian, there is no well established definition of the
mixing time in this context; this is why we write ``mixing time''.\\
4. Part (ii) becomes a stronger statement when $\epsilon$ becomes larger.}
\end{remarks}

\medskip\noindent
One of the key steps in order to prove (ii) of Theorem \ref{thm:sub} 
is to prove that the mean squared displacement of 
the walker is at most linear on the time scale $1/\mu$ uniform in the size of the torus. 
This result, which is also of independent interest, is presented next.
Here and throughout the paper, $\dist(x,y)$ will denote the graph distance between two vertices 
$x$ and $y$ in a given graph.

\begin{theorem}\label{thm:subMeanSquaredDisplacement}
Fix $d$ and $p\in (0,p_\critical(\mathbb{Z}^{d}))$.
Then there exists
$C_{\ref{thm:subMeanSquaredDisplacement}}=C_{\ref{thm:subMeanSquaredDisplacement}}(d,p)<\infty$ 
so that for all $n$, for all $\mu$ and for all $t$, if $G=\ndtorus$,
we have that
\begin{equation}\label{eq:subMeanSquaredDisplacement}
\E{\dist(X_{\frac{t}{\mu}},X_0)^2}\le 
C_{\ref{thm:subMeanSquaredDisplacement}} (t\vee 1)
\end{equation}
when we start the full system in stationarity with $u\times \pi_p$.
\end{theorem}

\begin{remark}{\rm
The above inequality is false if the ``$\vee 1$'' is removed since if $t=\mu$ is very 
small, then the LHS is not arbitrarily close to 0.
}\end{remark}

From Theorem \ref{thm:subMeanSquaredDisplacement}, we 
can obtain a similar bound for the full lattice $\mathbb{Z}^{d}$.

\begin{corollary}\label{cor:subMeanSquaredDisplacement}
Fix $d$ and $p\in (0,p_\critical(\mathbb{Z}^{d}))$. Then 
for all $\mu$ and for all $t$, if $G=\bfZ$, we have that
\[
\E{\dist(X_{\frac{t}{\mu}},0)^2}\le C_{\ref{thm:subMeanSquaredDisplacement}} 
(t\vee 1)
\]
when we start the system with distribution $\delta_0\times \pi_p$ and where
$C_{\ref{thm:subMeanSquaredDisplacement}}$ comes from
Theorem \ref{thm:subMeanSquaredDisplacement}.
\end{corollary}

\medskip\noindent
In Theorem \ref{thm:sub}(ii), Theorem~\ref{thm:subMeanSquaredDisplacement}
and Corollary~\ref{cor:subMeanSquaredDisplacement}, it was assumed that the bond configuration
was started in stationarity (in Theorem \ref{thm:sub}(ii), this is true since this was 
incorporated into the definition of $\tmixRW(\epsilon)$). For 
Theorem~\ref{thm:subMeanSquaredDisplacement} and Corollary~\ref{cor:subMeanSquaredDisplacement}, 
if the initial bond configuration is identically 1, $\mu=\frac{1}{n^{d+2}}$ and 
$t=\frac{1}{n^{d+1}}$, then the LHS's of these results grow to $\infty$ while 
the RHS's stay bounded and hence these results no longer hold. The reason for this is that the bonds
which the walker might encounter during this time period are unlikely to refresh
and so the walker is just doing ordinary random walk on $\bfZ$ or $\ndtorus$.
For similar reasons, if one takes the
initial bond configuration to be identically 1 in the definition of 
$\tmixRW(\epsilon)$, then if $\mu$ is very small, 
$\tmixRW(\epsilon)$ will be of the much smaller order $n^2$. 
However, due to Theorem \ref{thm:sub}(i), one cannot on the other hand
make $\tmixRW(\epsilon)$  larger than order $\frac{n^2}{\mu}$
by choosing an appropriate initial bond configuration.

For general $p$, we obtain the following lower bounds on the mixing time.
This is only of interest in the supercritical and critical cases $p\ge p_\critical$
since Theorem \ref{thm:sub}(ii) essentially implies this in the subcritical case; one minor
difference is that in (i) below, the constants do not depend on $p$.

\begin{theorem}\label{thm:sup}
(i). Given $d\ge 1$ and $\epsilon< 1$, there exist
$C_1>0$ and $n_0>0$
such that, for all $p$, for all $n\ge n_0$ and for all $\mu$, if $G=\dntorus$, then
\[
\tmixRW(\epsilon)\geq C_1 n^2.
\]
(ii). Given $d\ge 1$, $p$ and $\epsilon< 1-\theta_d(p)$, there exists $C_2>0$ 
and $n_0>0$ 
such that, for all $n\ge n_0$ and for all for $\mu$, if $G=\dntorus$, then
\begin{equation}\label{eq:SecondLowerBoundInSuper}
\tmixRW(\epsilon)\geq \frac{C_2}{\mu}.
\end{equation}
In particular, for $\epsilon< 1-\theta_d(p)$, we get a lower bound for
$\tmixRW(\epsilon)$ of order $\frac{1}{\mu}+ n^2$.
\end{theorem}

\begin{remark}{\rm
The lower bound in (ii) holds only for sufficiently small $\epsilon\in(0,1)$ depending on $p$.
To see this, take $d=2$ and choose $p$ sufficiently close to $1$ so 
that $\theta_2(p)>.999$. Take $n$ large, $\mu \ll n^{-4}$ 
and $t=Cn^2$, where $C$ is a large enough constant. 
Then, with probability going to 1 with $n$, the giant cluster
at time 0 will contain at least $.999$ fraction of the vertices. Therefore,
with probability about $.999$, the origin will be contained in this giant cluster.
By \cite{BM03}, with probability going to 1 with $n$, this giant cluster will have a 
mixing time of order $n^2$. Therefore, if $C$ is large,
a random walk on this giant cluster run for $t=Cn^2$ units of time will
be, in total variation, within $.0001$ of the uniform distribution on this cluster and hence
be, in total variation, within $.0002$ of the uniform distribution $u$. Since
$\mu \ll n^{-4}$, no edges will refresh up to time $t$ with very high probability and hence 
$\tv{\calL(X_t)}{u} \leq .0002$. Since $t \ll \frac{1}{\mu}$, we see that 
(\ref{eq:SecondLowerBoundInSuper}) above is not true for all $\epsilon$ but rather only 
for small $\epsilon$ depending on $p$. This strange dependence of the mixing time on 
$\epsilon$ cannot occur for a Markov process but can only occur here since
$\{X_t\}_{t\ge 0}$ is not Markovian.
}\end{remark}

We mention that heuristics suggest that the lower bound of $\frac{1}{\mu}+ n^2$
for the supercritical case should be the correct order.

We now give an analogue of Theorem \ref{thm:subMeanSquaredDisplacement} for general $p$.
This is also of interest in itself and as before is a key step in proving 
Theorem \ref{thm:sup}(i). While it is of course very similar to 
Theorem \ref{thm:subMeanSquaredDisplacement}, the fundamental difference between this result and
the latter result is that we do not now obtain linear mean squared displacement on the time scale
$1/\mu$ as we had before.

\begin{theorem}\label{thm:supMeanSquaredDisplacement}
Fix $d\ge 1$. Then there exists 
$C_{\ref{thm:supMeanSquaredDisplacement}}=C_{\ref{thm:supMeanSquaredDisplacement}}(d)$ so that 
for all $p$, for all $n$, for all $\mu$ and for all $t$, if $G=\dntorus$, then
\begin{equation}\label{eq:thm:supMeanSquaredDisplacement}
\E{\dist(X_{t},X_0)^2}\le C_{\ref{thm:supMeanSquaredDisplacement}} t
\end{equation}
when we start the full system in stationarity with $u\times \pi_p$.
\end{theorem}

From this, we can obtain, as before, a similar bound on the full lattice $\mathbb{Z}^{d}$.

\begin{corollary}\label{cor:supMeanSquaredDisplacement}
Fix $d\ge 1$. For all $p$, for all $\mu$ and for all $t$, if $G=\bfZ$, we have 
\begin{equation}\label{eq:cor:supMeanSquaredDisplacement}
\E{\dist(X_t,0)^2}\le C_{\ref{thm:supMeanSquaredDisplacement}} t
\end{equation}
when we start the full system with distribution $\delta_0\times \pi_p$ and where
$C_{\ref{thm:supMeanSquaredDisplacement}}$ comes from 
Theorem \ref{thm:supMeanSquaredDisplacement}.
\end{corollary}

\begin{remark}{\rm
Theorem \ref{thm:supMeanSquaredDisplacement} and Corollary
\ref{cor:supMeanSquaredDisplacement} are false if we start the bond configuration
in an arbitrary configuration. In \cite{BarlowPerkins89}, a subgraph $G$ of $\mathbb{Z}^{2}$
is constructed such that if one runs random walk on it, the expected mean squared distance 
to the origin at time $t^2$ is much larger than $t$ for large $t$. If we start the bond 
configuration in the state $G$ and $\mu$ is sufficiently small, then clearly 
(\ref{eq:cor:supMeanSquaredDisplacement}) fails for large $t$ provided (for example)
that $t^{d+1}\mu=o(1)$. Similarly, (\ref{eq:thm:supMeanSquaredDisplacement}) fails
for such $t$ and large $n$.
}\end{remark}

For Markov chains, one is often interested in studying hitting times. For discrete time
random walk on the torus $\ndtorus$, it is known that the maximum expectation of the time
it takes to hit a point
from an arbitary starting point behaves, up to constants, as $n^d$ for $d\ge 3$,
$n^2\log n$ for $d=2$ and $n^2$ for $d=1$. Here we obtain an analogous result for our dynamical
model in the subcritical regime.

For $y\in\ndtorus$, let $\sigma_y:=\inf\{t\ge 0:X_t=y\}$ be the first hitting time of $y$.
We will always start our system with distribution $\delta_0\times \pi_p$ (otherwise the results
would be substantially different). Finally, we let 
$H_d(n):=\max\{\E{\sigma_y}:y\in\dntorus\}$ denote the maximum expected hitting time.

\begin{theorem}\label{thm:hitting} 
For all $d\ge 1$ and $p\in (0,p_\critical(\mathbb{Z}^{d}))$, there exists
$C_{\ref{thm:hitting}}=C_{\ref{thm:hitting}}(d,p)<\infty$ so that the following hold.\\
(i). For all $n$ and for all $\mu\le 1$,
\[
 \frac{n^2}{C_{\ref{thm:hitting}}\,\mu} \le H_1(n) \le \frac{C_{\ref{thm:hitting}}\,n^2}{\mu}.
\]
(ii). For all $n$ and for all $\mu\le 1$,
\[
 \frac{n^2\log n}{C_{\ref{thm:hitting}}\,\mu} \le H_2(n) \le \frac{C_{\ref{thm:hitting}}\,n^2\log n}{\mu}.
\]
(iii). For all $d\ge 3$, for all $n$ and for all $\mu\le 1$,
\[
\frac{n^d}{C_{\ref{thm:hitting}}\,\mu} \le H_d(n) \le \frac{C_{\ref{thm:hitting}}\,n^d}{\mu}.
\]
\end{theorem}

\medskip\noindent
One of the usual methods for obtaining hitting time results (see \cite{LPW}) is to first develop
and then to apply results from electrical networks. However, in our case, where the network itself is
evolving in time, this approach does not seem to be applicable. More generally, many of the standard methods
for analyzing Markov chains do not seem helpful in studying the case, such as this,
where the transition probabilities are evolving in time stochastically.

\medskip\noindent
{\bf Previous work.}\\
Studying random walk in a random environment has been done since the early 1970's.
In the initial models studied, one chose a random environment which
would then be used to give the transition probabilities for a random walk. Once chosen, this 
environment would be fixed. There are many papers on random walk in 
random environment, far too many to list here.

After this, one studied random walks in an evolving random environment. 
The evolving random environment could be of a quite general nature. A sample of papers from
this area are 
\cite{Avena12}, \cite{AHR09}, \cite{AHRLLN}, \cite{ASV12}, \cite{BZ06}, \cite{BMP97}, \cite{DKL08}
and \cite{HSS12}.

However, the focus of these papers and the questions addressed in them are of a very different 
nature than the focus and questions addressed in the present paper. We therefore do not 
attempt to describe at all the results in these papers.

\noindent
{\bf Organization.}\\
The rest of the paper is organized as follows. In Section \ref{sec:Background}, 
various background will be given. 
In Section \ref{sec:rec}, Theorem \ref{thm:rec} is proved as well
as a central limit theorem and a general technical lemma
which will be used here as well as later on.
In Section \ref{sec:LowerBoundSub}, we prove the mean squared displacement results
and the lower bound on the mixing time in the subcritical regime:
Theorem \ref{thm:subMeanSquaredDisplacement}, Corollary \ref{cor:subMeanSquaredDisplacement} and
Theorem \ref{thm:sub}(ii).
In Section \ref{sec:LowerBoundGeneral}, we prove the mean squared displacement results
and the lower bound on the mixing time in the general case:
Theorem \ref{thm:supMeanSquaredDisplacement}, Corollary \ref{cor:supMeanSquaredDisplacement} and
Theorem \ref{thm:sup}. In Section \ref{sec:UpperBoundSub}, the upper bound on the mixing
time in the subcritical regime, Theorem \ref{thm:sub}(i), is proved. Finally, in 
Section \ref{sec:hitting}, Theorem \ref{thm:hitting} is proved. In
Section \ref{sec:question}, we state an open question.

\section{Background}\label{sec:Background}

In this section, we provide various background.

\noindent
{\bf Percolation.}\\
In percolation, one has a connected locally finite graph $G=(V,E)$ and a parameter
$p\in(0,1)$. One then declares each edge 
to be open (state 1) with probability $p$ and closed (state 0) with probability $1-p$
independently for different edges. Throughout this paper, we write $\pi_p$ 
for the corresponding product measure. One then studies 
the structure of the connected components (clusters) of the resulting 
subgraph of $G$ consisting of all vertices and all open edges. 
We will use $\bfP_{G,p}$ to denote probabilities when we perform $p$-percolation on
$G$. If $G$ is infinite,
the first question that can be asked is whether an infinite connected component 
exists (in which case we say {\em percolation} occurs). 
Writing $\calC$ for this latter event, Kolmogorov's 0-1 law tells us that
the probability of $\calC$ is, for fixed $G$ and $p$, either 0 or 1.
Since $\bfP_{G,p}(\calC)$ is nondecreasing in $p$, there
exists a critical probability $p_c=p_c(G)\in[0,1]$ such that
\[
\bfP_{G,p}(\calC)=\left\{
\begin{array}{ll}
0 & \mbox{for } p<p_c \\
1 & \mbox{for } p>p_c.
\end{array} \right.
\]
For all $x\in V$, let $C(x)$ denote the connected component of $x$, i.e.\ the set of vertices
having an open path to $x$. Finally, we let 
$\theta_d(p):=\bfP_{\mathbb{Z}^{d},p}(|C(0)|=\infty)$ where $0$ denotes the origin of
$\mathbb{Z}^{d}$. See \cite{Grimmett} for a comprehensive study of percolation.

Let $\mathbb{T}^{d,n}$ be the $d$-dimensional discrete torus with vertex set $[0,1,\ldots,n)^d$.
This is a transitive graph with $n^d$ vertices. For large $n$, the behavior of percolation on 
$\mathbb{T}^{d,n}$ is quite different depending on whether $p > p_c(\mathbb{Z}^d)$ or
$p < p_c(\mathbb{Z}^d)$; in this way the finite systems ``see'' the critical value for the infinite 
system. In particular, if $p > p_c(\mathbb{Z}^d)$, then with probability going to 1
with $n$, there will be a unique connected component with size of order $n^d$ 
(called the \emph{giant cluster}) 
while for $p < p_c(\mathbb{Z}^d)$, with probability going to 1 with $n$, all of the 
connected components will have size of order at most $\log(n)$. 

\noindent
{\bf Dynamical Percolation.}\\
This model was introduced in Section \ref{sec:Introduction}. Here we mention that the model
can equally well be described by having each edge of $G$ independently refresh its state
at rate $\mu$ and when it refreshes, it chooses to be in state 1 with probability $p$ and 
in state 0 with probability $1-p$ independently of everything else. We have already mentioned that
for all $G,p$ and $\mu$, the product measure $\pi_p$ is a stationary reversible 
probability measure for $\{\eta_t\}_{t\ge 0}$.
Dynamical percolation was introduced independently in \cite{HPS97} and by Itai Benjamini.
The types of questions that have been asked 
for this model is whether there exist exceptional times at which the 
percolation configuration looks markedly different from that at a fixed time.
See \cite{Steif09} for a recent survey of the subject. Our focus however in this paper
is quite different.

\noindent
{\bf Random walk on Dynamical Percolation.}\\
Random walk on Dynamical Percolation was introduced in Section \ref{sec:Introduction}. 
Throughout this paper, we will assume $\mu\le 1$. This model is most interesting when $\mu\to 0$
as the size of the graph gets large. Note that if $\mu=\infty$, then $\{X_t\}_{t\ge 0}$ 
would simply be ordinary simple random walk on $G$ with 
time scaled by $p$ and hence would not be interesting. One would similarly expect that
if $\mu$ is of order 1, the system should behave in various ways like ordinary random walk.
(We will see for example that the usual recurrence/transience dichotomy for random walk holds
in this model for fixed $\mu$.) This is why $\mu\to 0$ is the interesting regime.

\noindent
{\bf Mixing times for Markov chains.}\\
We recall the following standard definitions. Given two probability measures
$m_1$ and $m_2$ on a finite set $S$, we define the total variation distance between $m_1$ and 
$m_2$ to be
\[
\tv{m_1}{m_2} := \frac{1}{2} \sum_{s \in S}|m_1(s)-m_2(s)|.
\]
If $X$ and $Y$ are random variables, by $\tv{\calL(X)}{\calL(Y)}$, 
we will mean the total variation distance between their laws. 
There are other equivalent definitions; see \cite{LPW}, Section 4.2. One which we will need is that 
\[
\tv{\calL(X)}{\calL(Y)}=\inf\{\bfP(X'\neq Y')\}
\]
where the infimum is taken over all pairs of random variables $(X',Y')$ defined on the same space
where $X'$ has the same distribution as $X$ and $Y'$ has the same distribution as $Y$.

Given a continuous time finite state irreducible Markov chain with state space $S$, $t\ge 0$ and $x\in S$,
we let $P^t(x,\cdot)$ be the distribution of the chain at time $t$ when started in state $x$ and we
let $\pi$ denote the unique stationary distribution for the chain.

Next, one defines
\[
\tmix(\epsilon) := \inf\{t\geq 0 \colon \max_{x\in S} \tv{P^t(x,\cdot)}{\pi} \leq \epsilon\},
\]
in which case the standard definition of the mixing time of a chain, denoted by $\tmix$,
is $\tmix(1/4)$. It is well known (see \cite{LPW}, Section 4.5 for the discrete-time analogue)
that $\max_x \tv{P^t(x,\cdot)}{\pi}$ is decreasing in $t$ and that 
\begin{equation}\label{eq:MixingIterating}
\tmix(\epsilon)\le \lceil\log_2 \epsilon^{-1}\rceil \tmix.
\end{equation}

In the theory of mixing times, one typically has a sequence of Markov chains that one is 
interested in and one studies the limiting behavior of the corresponding sequence of mixing times.


\section{Recurrence/transience dichotomy}\label{sec:rec}

In this section, we prove Theorem~\ref{thm:rec} as well as a central limit theorem for the process
$\{X_t\}$.

\begin{proof}[{\bf Proof of Theorem~\ref{thm:rec}}]
Since $d$, $p$ and $\mu$ are fixed, we drop these superscripts.
We first prove this result when the bond configuration starts in state $\pi_p$; at the
end of the proof we extend this to a general initial bond configuration.

For this analysis, we let $\mathcal{F}_t$ be the $\sigma$-algebra 
generated by $\{M_s\}_{0\le s\le t}$ as well as, for each edge $e$,
the times before $t$ at which $e$ is refreshed and 
at which the random walker attempted to cross $e$.
We now define a sequence of sets $\{A_k\}_{k\ge 0}$. Let $A_0=\emptyset$. 
For $k\geq 1$, define $A_{k}$ to be the set of edges of $A_{k-1}$ that did not refresh during 
$[k-1,k]$ plus the set of edges that the walker attempted to cross during this interval of time
which did not refresh (during this time interval) after the last time 
(in this time interval) that the walker attempted to cross it. Note that
$A_k$ is measurable with respect to $\mathcal{F}_k$. Let $\tau_0=0$ and, for $k\geq 1$, define 
\[
\tau_k = \min\{i > \tau_{k-1} \colon A_i = \emptyset\}.
\]
We will see below that for all $k$, $\tau_k <\infty$ a.s.
Note that the random variables $\{\tau_k-\tau_{k-1}\}_{k\geq 1}$ are i.i.d.
For $k\geq 1$, let $U_k = X_{\tau_k}-X_{\tau_{k-1}}$. Clearly
the $\{U_k\}_{k\ge 1}$ are i.i.d.\ and hence $\{X_{\tau_n}\}_{n\ge 0}$ 
is a random walk on $\mathbb{Z}^d$ with step distribution $U_1$. It is 
easy to check that $U_1$ takes the value 0 as well as any of the $2d$ neighbors of 0 with 
positive probability and hence the random walk is fully supported, irreducible and
aperiodic. 

Let $J_i$ denote the number of attempted steps by the random walk during $[i-1,i]$
and $Z_k:=\sum_{i=\tau_{k-1}+1}^{\tau_{k}}J_i$ be the number of 
attempted steps by the random walk between $\tau_{k-1}$ and $\tau_k$. Now clearly $\{J_i\}_{i\ge 1}$
are i.i.d.\ as are $\{Z_k\}_{k\ge 1}$. A key step is to show that
\begin{equation}\label{eq:KeyExpTailForRecurrence}
\E{e^{cZ_1}}<\infty \mbox{ for some } c >0.
\end{equation}
(Note that this implies that each $\tau_k$ is finite a.s.)
Assuming (\ref{eq:KeyExpTailForRecurrence}) for the moment, we finish the proof of (i) when the bond 
configuration is started in stationarity. (\ref{eq:KeyExpTailForRecurrence})
implies, since $\dist(U_1,0)\le Z_1$, that $\dist(U_1,0)$ has an exponential tail and in particular 
a finite second moment. Since $U_1$ is obviously symmetric, it therefore
necessarily has mean zero. The fact that $\{X_{\tau_n}\}_{n\ge 0}$ is recurrent in one and two 
dimensions now follows from~\cite[Theorem~4.1.1]{LawlerLimic}. This proves (i) 
when the bond configuration is started in stationarity.

If $d\geq 3$, it follows from~\cite[Theorem~4.1.1]{LawlerLimic} that $\{X_{\tau_n}\}_{n\ge 0}$ 
is transient and so approaches $\infty$ a.s. To show (ii), we need to deal with times between
$\tau_{k}$ and $\tau_{k+1}$. Fix $M$, let $B_{\dist}(0,M)$ be the ball around 0 of $\dist$-radius $M$
and let $E_k$ be the event that the random walk returns to $B_{\dist}(0,M)$
during $[\tau_{k},\tau_{k+1}]$, we have
\[
\sum_{k=0}^\infty\pr{E_k} \le
\sum_{k=0}^\infty\pr{E_k\mid \dist(X_{\tau_k},0)\ge k^{\frac{1}{4d}}} +\pr{\dist(X_{\tau_k},0)
\le k^{\frac{1}{4d}}}.  
\]
Now $\pr{E_k\mid \dist(X_{\tau_k},0)\ge k^\frac{1}{4d}}\le \pr{Z_1\ge k^{\frac{1}{4d}}-M}$ 
and since $Z_1$ has an exponential
tail, the first terms are summable. Next, the local central limit theorem 
(cf.~\cite[Theorem~2.1.1]{LawlerLimic}) implies that for large $k$,
$\pr{\dist(X_{\tau_k},0)\le k^{\frac{1}{4d}}}\le \frac{Ck^{\frac{1}{4}}}{k^{\frac{d}{2}}}$ 
for some constant $C$. Since $d\ge 3$, it follows that the second terms are also summable. 
It now follows from Borel-Cantelli that the walker eventually leaves
$B_{\dist}(0,M)$ a.s.\ and hence by countable additivity (ii) holds 
when the bond configuration is started in stationarity.

We will now verify (\ref{eq:KeyExpTailForRecurrence}) by using Proposition \ref{prop:GoodProcessesDie}
with $Y_k=|A_k|+1$ and $\mathcal{F}_k$ being itself. Property (1) is immediate.
With $J_k$ as above and $R_k$ being the number of edges of $A_{k-1}$ that 
were refreshed during $[k-1,k]$, it is easy to see that
\[
|A_k| \leq |A_{k-1}| - R_k + J_k.
\]
This implies that
\[
      \Econd{|A_k|}{\mathcal{F}_{k-1}}
      \leq \Econd{|A_{k-1}| - R_k + J_k}{\mathcal{F}_{k-1}}
      = |A_{k-1}|e^{-\mu} + 1.
\]
This easily yields that there are positive numbers $a_0=a_0(\mu)$ and $b_0=b_0(\mu)<1$ so that
$\Econd{Y_k}{\mathcal{F}_{k-1}} \le b_0 Y_{k-1}$ on the event $Y_{k-1}>a_0$. 
This verifies property (2). Since properties (3) and (4) 
are easily checked, we can conclude from Proposition \ref{prop:GoodProcessesDie} at the end
of this section that 
$\tau_1$ has some positive exponential moment. An application of Lemma \ref{lem:NacuPeres05} to 
$\tau_1$ and the $J_i$'s allows us to conclude that
(\ref{eq:KeyExpTailForRecurrence}) holds. This completes the proof when the bond configuration
starts in stationarity.

We now analyze the situation starting from an arbitrary bond configuration.
Let $E_n$ be the event that some vertex whose $\dist$-distance to the 
origin is $n$ has an adjacent edge which does not refresh by time $\sqrt{n}$.
Let $H_n$ be the event that the number of attempted steps by the random walker 
by time $\sqrt{n}$ is larger than $n$. It is elementary to check that
\[
\sum_n  \Pruu{E_n}<\infty \mbox{ and } \sum_n  \Pruu{H_n} <\infty.
\]
By Borel-Cantelli, given $\epsilon>0$ there exists $n_0$ such that
\[
\Pruu{\cap_{n\ge n_0} E^c_n\cap \cap_{n\ge n_0} H^c_n}\ge 1-\epsilon.
\]
Now let $\eta$ be an arbitrary initial bond configuration and let $\eta^p$ be the random configuration
which is the same as $\eta$ at edges within distance $n_0$ of the origin and otherwise is chosen 
according to $\pi_p$. 

We claim that by what we have already proved, we can infer that
when the initial bond configuration is $\eta^p$, the random walker 
returns to 0 at arbitrarily large times a.s.\ if $d$ is 1 or 2 and converges to $\infty$ 
a.s.\ if $d\ge 3$. To see this, first observe that by Fubini's Theorem, we can infer from what
we have proved that for $\pi_p$-a.s.\ bond configuration, the random walker has the desired behavior
a.s. Therefore, since such a random bond configuration takes the same values as 
$\eta$ at edges within distance $n_0$ of the origin with positive probability,
it must be the case that a.s.\ $\eta^p$ is such that the random walk has the desired behavior.
Using Fubini's Theorem again demonstrates this claim.

One can next couple the random walker when the initial bond 
configuration is $\eta$ with the random walker when the initial bond configuration is $\eta^p$ in 
the obvious way. They will remain together provided $\cap_{n\ge n_0} E^c_n$ holds 
and $\cap_{n\ge n_0} H^c_n$ holds for the walks. Therefore the random walker with initial 
bond configuration $\eta$ has the claimed behavior
with probability $1-\epsilon$. As $\epsilon$ is arbitrary, we are done.
\end{proof}

We now provide a central limit theorem for the walker.

\begin{theorem}\label{thm:CLT}
Given $d,p$ and $\mu$, there exists $\sigma\in (0,\infty)$ so that
random walk in dynamical percolation on $\bfZ$ with parameters $p$ and $\mu$ 
started from an arbitrary configuration satisfies
\[
\left\{\frac{X_{kt}}{\sqrt{k}}\right\}_{t\in [0,1]} \Rightarrow \,\,\, \{B_t\}_{t\in [0,1]}
\]
in $C[0,1]$ as $k\to\infty$ where $\{B_t\}_{t\in [0,1]}$ is a standard $d$-dimensional 
Brownian motion with variance $\sigma^2$. Moreover
\[
\sigma^2=\frac{{\rm Var}(U^{(1)}_1)}{\E{\tau_1}}
\]
where $U_1$ and $\tau_1$ are given in the proof of Theorem \ref{thm:rec}
and $U_1^{(1)}$ is the first coordinate of $U_1$.  
\end{theorem}

\begin{proof}
This type of argument is very standard and so we only sketch the proof. We therefore only prove 
the convergence for the fixed value $t=1$ and we also assume that the bond configuration is started 
in stationarity with distribution $\pi_p$ and that the walker starts at the origin. To deal
with general initial bond configurations, the methods in the proof of Theorem \ref{thm:rec}
can easily be adapted.

Now, symmetry considerations give that $\E{U_1^{(1)}}=0$ and that the different coordinates are 
uncorrelated.
We have also seen that $U_1^{(1)}$ has a finite second moment. The central limit theorem now tells us that
\begin{equation}\label{eq:CLT}
\frac{X_{\tau_n}}{\sqrt{n}{\rm Var}(U^{(1)}_1)^{1/2}}=
\frac{\sum_{1}^n U_i}{\sqrt{n}{\rm Var}(U^{(1)}_1)^{1/2}} \Rightarrow \calN
\end{equation}
where $\calN$ is a standard $d$-dimensional Gaussian.

We need to show that 
$\frac{X_{k}}{\sqrt{k}}$ converges to the appropriate Gaussian.
Let $n(k):=\lceil k/\E{\tau_1}\rceil$ and write $\frac{X_{k}}{\sqrt{k}}$ as
\[
\frac{\sqrt{n(k)}}{\sqrt{k}} 
\left[
\frac{X_k-X_{n(k)\E{\tau_1}}}{\sqrt{n(k)}} +
\frac{X_{n(k)\E{\tau_1}}-X_{\tau_{n(k)}}}{\sqrt{n(k)}} +
\frac{X_{\tau_{n(k)}}}{\sqrt{n(k)}}
\right]\, .
\]
The first fraction converges to $1/\sqrt{\E{\tau_1}}$. The first fraction in the second factor
is easily shown to converge in probability to 0. The weak law of large numbers gives that
$\tau_{n(k)}/n(k)$ converges in probability to $\E{\tau_1}$ which easily leads to the 
second fraction in the second factor converging in probability to 0. Finally, using (\ref{eq:CLT})
for the last fraction, we obtain the result.
\end{proof}

\noindent
{\bf Technical lemma.}\\
We now present the technical lemma which was used in the previous proof and will be used
again later on.
This result is presumably well known in some form but we provide the proof nevertheless
for completeness. As the result is ``obvious'', the reader might choose to skip the proof.

\begin{proposition}\label{prop:GoodProcessesDie}
For all $\alpha\ge 1$, $\delta<1$, $\epsilon>0$ and $\gamma$, there exist 
$c_{\ref{prop:GoodProcessesDie}}=c_{\ref{prop:GoodProcessesDie}}(\alpha,\delta,\epsilon,\gamma)
>0$ 
and 
$C_{\ref{prop:GoodProcessesDie}}=C_{\ref{prop:GoodProcessesDie}}(\alpha,\delta,\epsilon,\gamma)<\infty$ 
with the following
property. If $\{Y_i\}_{i\ge 0}$ is a discrete time process taking 
values in $\{1,2,\ldots\}$ adapted to a filtration  $\{\mathcal{F}_i\}_{i\ge 0}$ satisfying \\
(1) $Y_0=1$, \\
(2) for all $i$, $\E{Y_{i+1}\mid \mathcal{F}_i}\le \delta\, Y_i$ on $Y_i > \alpha$, \\
(3) for all $i$, $\Pruu{Y_{i+1}=1\mid \mathcal{F}_i}\ge \epsilon$ on $Y_i\le \alpha$ and \\
(4) for all $i$, $\E{Y_{i+1}\mid \mathcal{F}_i}\le \gamma$ on $Y_i\le \alpha$, \\
and if $T:=\min\{i\ge 1: Y_i=1\}$, then
\[
\E{e^{c_{\ref{prop:GoodProcessesDie}}T}} \le C_{\ref{prop:GoodProcessesDie}}
\]
and so in particular by Jensen's inequality
\[
\E{T} \le \frac{\log C_{\ref{prop:GoodProcessesDie}}}{c_{\ref{prop:GoodProcessesDie}}}.
\]
\end{proposition}


To prove this, we will need a slight strengthening of a lemma from \cite{NacuPeres05} which essentially
follows the same proof.

\begin{lemma}\label{lem:NacuPeres05}
Given positive numbers $\lambda$ and $a$, there exist 
$c_{\ref{lem:NacuPeres05}}=c_{\ref{lem:NacuPeres05}}(\lambda,a)>0$ and
$C_{\ref{lem:NacuPeres05}}=C_{\ref{lem:NacuPeres05}}(\lambda,a)<\infty$ so that if
$\{X_i\}$ are nonnegative random variables adapted to a filtration  $\{\mathcal{G}_i\}_{i\ge 0}$ 
satisfying 
\begin{equation}\label{eq:LInfinityAssumption}
\|\E{e^{\lambda X_{i+1}}\mid \mathcal{G}_{i}}\|_{\infty} \le a\,\, \mbox{for all $i\ge 0$}, 
\end{equation}
and $M$ is a nonnegative integer valued random variable satisfying
\begin{equation}\label{eq:OtherAssumption}
\E{e^{\lambda M}}\le a,
\end{equation}
then
\begin{equation}\label{eq:DesiredInequality}
\E{e^{c_{\ref{lem:NacuPeres05}}\sum_{i=1}^M X_i}}\le C_{\ref{lem:NacuPeres05}}.
\end{equation}
(Note that this implies, by Jensen's inequality, that
$\E{\sum_{i=1}^M X_i}\le \frac{\log C_{\ref{lem:NacuPeres05}}}{c_{\ref{lem:NacuPeres05}}}$.)
\end{lemma}

\begin{proof}
Choose $k=k(\lambda,a)$ sufficiently large so that $a< e^{\lambda k}$.
We claim there exist $b<1$ and $B$, depending only on $\lambda$ and $a$,
such that for all $n$
\begin{equation}\label{eq:BoundedOnSum}
\Pruu{\sum_{i=1}^M X_i\ge kn}\le B b^n.
\end{equation}
To see this, note that this above probability is at most
\[
\Pruu{M \ge n}+\Pruu{\sum_{i=1}^n X_i\ge kn}
\le a e^{-\lambda n}+\frac{\E{e^{\lambda\sum_{i=1}^n X_i}}}{e^{\lambda k n}}
\]
by Markov's inequality and (\ref{eq:OtherAssumption}).
Using (\ref{eq:LInfinityAssumption}), taking conditional expectations and iterating, one sees that
\[
\E{e^{\lambda\sum_{i=1}^n X_i}}\le a^n
\]
and hence the second term is at most $(\frac{a}{e^{\lambda k}})^n$. We can conclude that there are $b$ 
and $B$, depending only on $\lambda$ and $a$, such that
(\ref{eq:BoundedOnSum}) holds. Since $b,B$ and $k$ all depend only on $\lambda$ and $a$,
it then easily follows from (\ref{eq:BoundedOnSum}) that there are $c_{\ref{lem:NacuPeres05}}$ 
and $C_{\ref{lem:NacuPeres05}}$ depending only on $\lambda$ and $a$ so that
(\ref{eq:DesiredInequality}) holds.
\end{proof}

\begin{proof}[{\bf Proof of Proposition \ref{prop:GoodProcessesDie}}]
Let $U_0=0$ and for $k\ge 1$, let
$U_{k}:=\min\{i\ge U_{k-1}+1: Y_i\le \alpha\}$. Let 
$M:=\min\{k\ge 0: Y_{U_k+1}=1\}$. Property (3) implies that $M$ is stochastically dominated by a 
geometric random variable with parameter $\epsilon$. Next, clearly
\begin{equation}\label{eq:TBoundedBySum}
T\le 1+\sum_{k=1}^M (U_{k}-U_{k-1}).
\end{equation}
(Equality does not necessarily hold since it is possible that $Y_{T-1}>\alpha$ in which case 
$T-1$ does not correspond to any $U_k$.)
We claim that for all $k\ge 1$
\begin{equation}\label{eq:LinfinityBound}
\|\E{\left(\frac{1}{\delta}\right)^{U_{k}-U_{k-1}}\mid \mathcal{F}_{U_{k-1}}}\|_{\infty} \le 
\frac{\gamma}{\delta}.
\end{equation}
We write 
\[
\E{\left(\frac{1}{\delta}\right)^{U_{k}-U_{k-1}}\mid \mathcal{F}_{U_{k-1}}}=
\E{\E{\left(\frac{1}{\delta}\right)^{U_{k}-U_{k-1}}\mid \mathcal{F}_{U_{k-1}+1}}\mid \mathcal{F}_{U_{k-1}}}.
\]
Concerning the inner conditional expectation, we claim that 
\begin{equation}\label{eq:InnerExpectation}
\E{\left(\frac{1}{\delta}\right)^{U_{k}-U_{k-1}}\mid \mathcal{F}_{U_{k-1}+1}}\le
\frac{Y_{U_{k-1}+1}}{\delta}.
\end{equation}

Case 1: $Y_{U_{k-1}+1}\le \alpha$.  \\
In this case, $U_{k}=U_{k-1}+1$ and so 
\[
\E{\left(\frac{1}{\delta}\right)^{U_{k}-U_{k-1}}\mid \mathcal{F}_{U_{k-1}+1}}=\frac{1}{\delta}.
\]

Case 2: $Y_{U_{k-1}+1}> \alpha$.  \\
In this case, we first make the important observation that property (2)
implies that on the event $Y_{U_{k-1}+1}> \alpha$,
\[
M_j:=\left(\frac{1}{\delta}\right)^{j\wedge (U_k-U_{k-1}-1)}
Y_{(j\wedge (U_k-U_{k-1}-1))+U_{k-1}+1}, \,\,\, j\ge 0
\]
is a supermartingale with respect to $\{\mathcal{F}_{j+U_{k-1}+1}\}_{j\ge 0}$.
From the theory of nonnegative supermartingales (see \cite{Durrett}, Chapter 5), we can let $j\to\infty$
in the defining inequality of a supermartingale and conclude that 
\[
\E{\left(\frac{1}{\delta}\right)^{U_k-U_{k-1}-1} Y_{U_k}
\mid \mathcal{F}_{U_{k-1}+1}}\le Y_{U_{k-1}+1},
\]
It follows that
\begin{equation}\label{eq:StoppingTimeInequality}
\E{\left(\frac{1}{\delta}\right)^{U_k-U_{k-1}}\mid \mathcal{F}_{U_{k-1}+1}}
\le \frac{Y_{U_{k-1}+1}}{\delta}.
\end{equation}
Since the $Y_i$'s are at least 1, this establishes (\ref{eq:InnerExpectation}).

Taking the conditional expectation of the two sides of (\ref{eq:InnerExpectation})
with respect to $\mathcal{F}_{U_{k-1}}$ and using property (4) finally 
yields (\ref{eq:LinfinityBound}). 

Lastly, Lemma \ref{lem:NacuPeres05} (with $X_i=U_i- U_{i-1}$, 
$\mathcal{G}_i=\mathcal{F}_{U_i}$, $M=M$, $\lambda=\min\{\frac{\epsilon}{2},\log(\frac{1}{\delta})\}$
and $a=\max\{\frac{\gamma}{\delta},2\}$) together with
(\ref{eq:TBoundedBySum}), (\ref{eq:LinfinityBound}) and the fact that $M$ is dominated by a geometric
random variable with parameter $\epsilon$ gives us the desired conclusion.
\end{proof}

\begin{remark}{\rm 
We see that $c_{\ref{prop:GoodProcessesDie}}$ and $C_{\ref{prop:GoodProcessesDie}}$ actually
depend only on $\delta,\epsilon$ and $\gamma$ but not on $\alpha$.
}\end{remark}


\section{Proofs of the mixing time lower bound in the subcritical case}
\label{sec:LowerBoundSub}

In this section, we prove Theorem~\ref{thm:subMeanSquaredDisplacement},
Theorem \ref{thm:sub}(ii) and Corollary~\ref{cor:subMeanSquaredDisplacement}.
We begin with the proof of Theorem~\ref{thm:subMeanSquaredDisplacement}
as this will be used in the proof of Theorem \ref{thm:sub}(ii).

\begin{proof}[{\bf Proof of Theorem~\ref{thm:subMeanSquaredDisplacement}}]
Fix $d$ and $p\in (0,p_\critical(\mathbb{Z}^{d}))$.
Choose $\beta=\beta(d,p)$ sufficiently small
so that for all $n$ and $\mu$, the probability that, for $\{\eta_t\}_{t\ge 0}$, 
a fixed edge $e$ is open at some point in 
$[0,\beta/\mu]$ is less than $p_\critical(\mathbb{Z}^{d})$. By time scaling, the latter probability
is independent of $\mu$ (and of course of $n$).

Let $g_n:\ndtorus\rightarrow \mathbb{R}^{2d}$ be given by
\begin{equation}\label{eq:Lipschitz}
g_n(x_1,\ldots,x_d):=(n\cos(2\pi x_1/n),n\sin(2\pi x_1/n),\ldots,n\cos(2\pi x_d/n),n\sin(2\pi x_d/n)).
\end{equation}

Observe that for fixed $d$, the functions $\{g_n\}_{n\ge 1}$ are uniformly bi--Lipschitz when 
$\ndtorus$ is equipped with the metric $\dist$ and $\mathbb{R}^{2d}$ has its usual metric.
Let $C_{\mathrm{Lip}}=C_{\mathrm{Lip}}(d)$ be a uniform bound on the bi--Lipschitz constants.

We need the following two lemmas. The first lemma will be proved afterwards while the second 
lemma, which is implicitly contained in \cite{Ball92}, is stated explicitly in \cite{NPSS06}; 
in fact a strengthening of it yielding a maximal version is proved in \cite{NPSS06}.

\begin{lemma}\label{lem:sub.1StepBound}
There exists $C_{\ref{lem:sub.1StepBound}}=C_{\ref{lem:sub.1StepBound}}(d,p)$ so that 
for all $n$, for all $\mu$ and for all $s\le\beta$, 
\[
\E{\dist(X_{\frac{s}{\mu}},X_0)^2}\le C_{\ref{lem:sub.1StepBound}}
\]
when we start the full system in stationarity.
\end{lemma}

\begin{lemma}\label{lem:Ball}
Let $\{Y_i\}_{i\in \mathbb{Z}}$ be a discrete time stationary reversible Markov chain with finite state space $S$ and let 
$h:S\rightarrow \mathbb{R}^m$. Then for each $k\ge 1$,
\[
\E{\|h(Y_k)-h(Y_0)\|_{L^2}^2}\le k\E{\|h(Y_1)-h(Y_0)\|_{L^2}^2}
\]
where 
$\|\,\|_{L^2}$ denotes the Euclidean norm on $\mathbb{R}^m$.
\end{lemma}

We may assume that $\beta \le 1$. For $t\le \beta (\le 1)$, the LHS of
(\ref{eq:subMeanSquaredDisplacement}) is by Lemma \ref{lem:sub.1StepBound} at most 
$C_{\ref{lem:sub.1StepBound}}$ which is at most 
$C_{\ref{thm:subMeanSquaredDisplacement}} (t\vee 1)$ if 
$C_{\ref{thm:subMeanSquaredDisplacement}}$ is taken to be larger than
$C_{\ref{lem:sub.1StepBound}}$.

On the other hand, if $t\ge \beta$, choose $\ell\in \mathbb{N}$ so that 
$v:=\frac{t}{\ell}\in [\beta/2,\beta]$. Consider the discrete time finite state stationary 
reversible Markov chain given by
\[
Y_k:=M_{\frac{kv}{\mu}}, \,\,\, k\in \mathbb{Z}
\]
with state space $S:=\ndtorus\times\{0,1\}^{E(\ndtorus)}$. With all the parameters for the chain fixed, let
$h_n:S\rightarrow \mathbb{R}^{2d}$ be given by $h_n(x,\eta):=g_n(x)$. Then
Lemma \ref{lem:sub.1StepBound} (with $s=v$) together with the uniform bi--Lipschitz property
of the $g_n$'s implies that 
\[
\E{\|h_n(Y_1)-h_n(Y_0)\|_{L^2}^2}\le C^2_{\mathrm{Lip}}C_{\ref{lem:sub.1StepBound}}.
\]

We now can apply Lemma \ref{lem:Ball} with $k=\ell$ and obtain
\[
\E{\|g_n(X_{\frac{t}{\mu}})-g_n(X_{0})\|_{L^2}^2}
\le C^2_{\mathrm{Lip}}C_{\ref{lem:sub.1StepBound}}\ell.
\]
Since $\frac{t}{\ell}\in [\beta/2,\beta]$, we have that $\ell\le 2t/\beta$. Using this
and the bi--Lipschitz property of the $g_n$'s again, we obtain
\[
\E{\dist(X_{\frac{t}{\mu}},X_{0})^2}
\le 2C^4_{\mathrm{Lip}}C_{\ref{lem:sub.1StepBound}}t/\beta.
\]
As all of the terms except $t$ on the RHS only depend on $d$ and $p$, we are done.
\end{proof}

The proof of Lemma~\ref{lem:sub.1StepBound} requires the following important result
concerning subcritical percolation. 
For  $V'\subseteq V$, we let $\mathrm{Diam}(V'):=\max\{\dist(x,y):x,y\in V'\}$
denote the diameter of $V'$. This following theorem is Theorem~5.4 in \cite{Grimmett} in the case of 
$\mathbb{Z}^{d}$. The statement for $\mathbb{Z}^{d}$ immediately implies the result for $\dntorus$.

\begin{theorem}\label{thm:exp.radius}
For any $d\ge 1$ and $\alpha\in (0,p_\critical(\mathbb{Z}^{d}))$, there exists 
$C_{\ref{thm:exp.radius}}=C_{\ref{thm:exp.radius}}(d,\alpha)>0$ so that for all $r\ge 1$,
\[
\bfP_{\mathbb{Z}^{d},\alpha}({\mathrm{Diam}}(C(0))\ge r)\le e^{-C_{\ref{thm:exp.radius}}r}.
\]
The previous line holds with $\mathbb{Z}^{d}$ replaced by $\dntorus$.
\end{theorem}

We now give the 

\begin{proof}[{\bf Proof of Lemma~\ref{lem:sub.1StepBound}}]
Let $\overline{\eta}$ be the set of edges which are open some time during $[0,\beta/\mu]$.
By our choice of $\beta$, there exists $p_0=p_0(d,p)< p_c(\mathbb{Z}^d)$ so that for all
$n$ and all $\mu$, the distribution of $\overline{\eta}$ is $\pi_{p_0}$.

Letting $C_{\overline{\eta}}(x)$ denote the cluster of $x$ with respect to the bond configuration
$\overline{\eta}$, the observation above and 
Theorem \ref{thm:exp.radius} implies that there exists a constant
$C_{\ref{lem:sub.1StepBound}.1}=C_{\ref{lem:sub.1StepBound}.1}(d,p)$ so that for all $n$, for all $\mu$
and for all $x\in\ndtorus$,
\[
\E{(\mathrm{Diam}(C_{\overline{\eta}}(x)))^2}\le C_{\ref{lem:sub.1StepBound}.1}.
\]
By independence of $X_{0}$ and 
${\overline{\eta}}$, we get  
\begin{equation}\label{eq:RadiusSquaredBound}
\E{(\mathrm{Diam}(C_{\overline{\eta}}(X_{0})))^2}\le C_{\ref{lem:sub.1StepBound}.1}.
\end{equation}
Since the random walker can only move along
${\overline{\eta}}$ during $[0,\beta/\mu]$, we have that for all $s\le \beta$,
$X_{\frac{s}{\mu}}$ necessarily belongs to 
$C_{\overline{\eta}}(X_{0})$ and hence
$\dist(X_{\frac{s}{\mu}},X_0)\le 
\mathrm{Diam}(C_{\overline{\eta}}(X_{0}))$. The result now follows from 
(\ref{eq:RadiusSquaredBound}).  
\end{proof}                     

We next move to the 

\begin{proof}[{\bf Proof of Corollary~\ref{cor:subMeanSquaredDisplacement}}]
Fix $\mu$ and $t$. By Theorem~\ref{thm:subMeanSquaredDisplacement} and symmetry, 
we have that for each $n$,
\[
\E{\dist(X_{\frac{t}{\mu}},0)^2|\delta_0\times \pi_p}\le 
C_{\ref{thm:subMeanSquaredDisplacement}} (t \vee 1).
\]
Clearly $\dist(X_{\frac{t}{\mu}},0)|\delta_0\times \pi_p$ converges in 
distribution to  $\dist(X_{\frac{t}{\mu}},0)$ as $n\to\infty$ where the latter is started
in $\delta_0\times \pi_p$. The result now follows by squaring and applying Fatou's lemma.
\end{proof}

\begin{proof}[{\bf Proof of Theorem~\ref{thm:sub}(ii)}]
Fix $d$, $p\in (0,p_\critical(\mathbb{Z}^{d}))$ and $\epsilon<1$. It suffices to show that
there exists $\delta=\delta(d,p,\epsilon)>0$ and $n_0=n_0(d,p,\epsilon)>0$ so that for $n\ge n_0$
and $s\le \frac{\delta n^2}{\mu}$
\[
\tv{\calL(X_s)}{u} > \epsilon.
\]
By symmetry, the distribution of 
$\dist(X_{\frac{t}{\mu}},X_0)$ conditioned on
$\{X_0=a\}$ does not depend on $a$. Hence by Theorem \ref{thm:subMeanSquaredDisplacement}
and Markov's inequality, we have that for all $\lambda >0$, for all $n$,
for all $\mu$ and for all $t$
\begin{equation}\label{eq:1StepBoundUsingMarkov}
\pr{\dist(X_{\frac{t}{\mu}},0)\ge \lambda|\delta_0\times \pi_p}
\le C_{\ref{thm:subMeanSquaredDisplacement}}(t\vee 1)/\lambda^2
\end{equation}
where $C_{\ref{thm:subMeanSquaredDisplacement}}$ comes from 
Theorem \ref{thm:subMeanSquaredDisplacement}. Next, choose $b=b(d,\epsilon)>0$ so that
\[
(2b)^d< \frac{1-\epsilon}{2}.
\]
We then have that there exists $n_0=n_0(d,p,\epsilon)>0$ sufficiently large
so that for all $n\ge n_0$ we have that
\begin{equation}\label{eq:GeometricBound}
|\{x\in\ndtorus:\dist(x,0)\le bn\}|\le \frac{(1-\epsilon)n^d}{2}
\end{equation}
and
\[
\frac{C_{\ref{thm:subMeanSquaredDisplacement}}}{(bn)^2} < \frac{1-\epsilon}{2}.
\]
Next choose $\delta=\delta(d,p,\epsilon)>0$ so that
\[
\frac{C_{\ref{thm:subMeanSquaredDisplacement}}\delta}{b^2}<  \frac{1-\epsilon}{2}.
\]

We now let $n\ge n_0$ and $s\le \frac{\delta n^2}{\mu}$. Applying
(\ref{eq:1StepBoundUsingMarkov}) with $t=s\mu \le \delta n^2$ and
$\lambda=bn$ yields
\begin{equation}\label{eq:1StepBoundUsingMarkovAGAIN}
\pr{\dist(X_{s},0)\ge bn|\delta_0\times \pi_p}
\le \frac{C_{\ref{thm:subMeanSquaredDisplacement}}(\delta n^2\vee 1)}{(bn)^2}
< \frac{1-\epsilon}{2}.
\end{equation}
Letting $E_n:=\{x\in\ndtorus:\dist(x,0)\le bn\}$, we have by (\ref{eq:1StepBoundUsingMarkovAGAIN})
\[
\pr{X_s\in E_n|\delta_0\times \pi_p} \ge \frac{1+\epsilon}{2}
\]
and by (\ref{eq:GeometricBound}), we have
\[
u(E_n)\le \frac{1-\epsilon}{2}.
\]
Hence, by considering the set $E_n$, it follows that
$\tv{\calL(X_s)}{u} \ge \frac{1+\epsilon}{2}-\frac{1-\epsilon}{2}=\epsilon$, completing the proof.
\end{proof}

We end this section by proving that not only is the mixing time for the full system of order 
at least $n^2/\mu$ but that this is also a lower bound on the relaxation time. Moreover, in proving this, 
we will only use Lemma \ref{lem:sub.1StepBound} and do not need to appeal to the so-called 
Markov type inequality contained in Lemma \ref{lem:Ball}. 

\begin{proposition}\label{prop:relaxation}
For any $d$ and $p\in (0,p_\critical(\mathbb{Z}^{d}))$, there exists 
$C_{\ref{prop:relaxation}}=C_{\ref{prop:relaxation}}(d,p)>0$
such that, for all $n$ and for all $\mu$, the relaxation time of the full system is at least
$\frac{C_{\ref{prop:relaxation}}n^2}{\mu}$.
\end{proposition}

\begin{remarks}{\rm While it is a general fact that the mixing time is bounded below by (a universal
constant times) the relaxation time, this does not provide an alternative proof of Theorem 
\ref{thm:sub}(ii) for two reasons. First, in the latter, we have a lower bound for the ``mixing time'' 
of the walker (which is stronger than just having a lower bound on the mixing time for the full system) and 
secondly $\epsilon$ in Theorem \ref{thm:sub}(ii) can be taken close to 1 while one could only conclude
this for $\epsilon < \frac{1}{2}$ directly from a lower bound on the relaxation time.}
\end{remarks}

\begin{proof}
We will obtain an upper bound on the spectral gap by considering the usual Dirichlet form;
see Section 13.3 in \cite{LPW}. Consider the function $f_n:\ndtorus\times\{0,1\}^{E(\ndtorus)}
\rightarrow \mathbb{R}$ given by $f_n(x,\eta):=\dist(x,0)$ where $0$ is the origin in $\ndtorus$.
Clearly, there exists a constant $C_{\ref{prop:relaxation}.1}=C_{\ref{prop:relaxation}.1}(d)>0$ 
such that for all $d,p,n$ and $\mu$,
\[
{\rm Var}(f_n)\ge C_{\ref{prop:relaxation}.1} n^2
\]
where ${\rm Var}(f_n)$ denotes the variance of $f_n$ with respect to the stationary distribution. 

Letting $\beta$ be defined as in the proof of Theorem~\ref{thm:subMeanSquaredDisplacement},
Lemma \ref{lem:sub.1StepBound} and the triangle inequality imply that
\[
\E{|f_n(M_{\frac{\beta}{\mu}})-f_n(M_0)|^2}\le 4 C_{\ref{lem:sub.1StepBound}}.
\]
Hence
\[
\frac{\E{|f_n(M_{\frac{\beta}{\mu}})-f_n(M_0)|^2}}{2{\rm Var}(f_n)}
\le \frac{2C_{\ref{lem:sub.1StepBound}}}{C_{\ref{prop:relaxation}.1} n^2}.
\]
By Section 13.3 in \cite{LPW}, we conclude that the spectral gap for the discrete time process 
viewed at times $0,\frac{\beta}{\mu},\frac{2\beta}{\mu},\ldots$ is at most
$\frac{2C_{\ref{lem:sub.1StepBound}}}{C_{\ref{prop:relaxation}.1} n^2}$.  If 
$-\lambda=-\lambda(d,p,n,\mu)$ is the nonzero
eigenvalue of minimum absolute value for the infinitesimal generator of the continuous time process 
(in which case $\lambda$ is the spectral gap for the continuous time process), 
then the spectral gap for the above discrete time process is $1-e^{\frac{-\lambda\beta}{\mu}}$ and so
\[
1-e^{\frac{-\lambda\beta}{\mu}}\le 
\frac{2C_{\ref{lem:sub.1StepBound}}}{C_{\ref{prop:relaxation}.1} n^2}.
\]
We can conclude that for large $n$, for any $\mu$, we have that
$\frac{\lambda\beta}{\mu}\le \frac{1}{2}$. Since $1-e^{-x}\ge x/2$ on $[0,1]$, we conclude
that 
\[
\frac{\lambda\beta}{2\mu}\le 
\frac{2C_{\ref{lem:sub.1StepBound}}}{C_{\ref{prop:relaxation}.1} n^2}
\]
or
\[
\lambda\le \frac{4C_{\ref{lem:sub.1StepBound}}\mu}{\beta C_{\ref{prop:relaxation}.1} n^2}.
\]
Since the relaxation time is the reciprocal of the spectral gap, we are done.
\end{proof}


\section{Proofs of the mixing time lower bounds in the general case}
\label{sec:LowerBoundGeneral}

In this section, we prove Theorem~\ref{thm:supMeanSquaredDisplacement},
Theorem \ref{thm:sup} and Corollary~\ref{cor:supMeanSquaredDisplacement}.
We begin with the proof of Theorem~\ref{thm:supMeanSquaredDisplacement}
as this will be used in the proof of Theorem \ref{thm:sup}.

\begin{proof}[{\bf Proof of Theorem~\ref{thm:supMeanSquaredDisplacement}}]
Clearly, $\dist(X_{a},X_0)$ is stochastically dominated by a Poisson random variable
with parameter $s$. It follows that
\begin{equation}\label{eq:PoissonBound}
\E{\dist(X_{s},X_0)^2}\le s+s^2.
\end{equation}
(This will be used in the same way that Lemma \ref{lem:sub.1StepBound} was used.)
(\ref{eq:PoissonBound}) tells us that (\ref{eq:thm:supMeanSquaredDisplacement}) holds
for $t\le 1$ if $C_{\ref{thm:supMeanSquaredDisplacement}}\ge 2$.

If $t\ge 1$, choose $\ell\in \mathbb{N}$ so that $v:=\frac{t}{\ell}\in [1/2,1]$. Consider the discrete 
time finite state stationary reversible Markov chain given by
\[
Y_k:=M_{kv}, \,\,\, k\in \mathbb{Z}.
\]
Letting $S$ and $h_n$ be as in the proof of Theorem~\ref{thm:subMeanSquaredDisplacement},
(\ref{eq:PoissonBound}) (with $s=v\le 1$) together with the uniform bi--Lipschitz property
of the $g_n$'s implies that 
\[
\E{\|h(Y_1)-h(Y_0)\|_{L^2}^2}\le 2 C^2_{\mathrm{Lip}}.
\]

Lemma \ref{lem:Ball} with $k=\ell$ now yields
\[
\E{\|g_n(X_t)-g_n(X_{0})\|_{L^2}^2} \le 2 C^2_{\mathrm{Lip}}\ell.
\]
Since $\frac{t}{\ell}\in [1/2,1]$, we have that $\ell\le 2t$. Using this
and the bi--Lipschitz property of the $g_n$'s again, we obtain
\[
\E{\dist(X_t,X_{0})^2}
\le 4C^4_{\mathrm{Lip}}t.
\]
Letting $C_{\ref{thm:supMeanSquaredDisplacement}}:=4C^4_{\mathrm{Lip}} (\ge 2)$, we obtain 
(\ref{eq:thm:supMeanSquaredDisplacement}) for $t\ge 1$ as well.
\end{proof}

\begin{proof}[{\bf Proof of Corollary~\ref{cor:supMeanSquaredDisplacement}}]
This can be obtained from Theorem~\ref{thm:supMeanSquaredDisplacement} in the 
exact same way as Corollary~\ref{cor:subMeanSquaredDisplacement} was obtained from
Theorem~\ref{thm:subMeanSquaredDisplacement}.
\end{proof}

\begin{proof}[{\bf Proof of Theorem~\ref{thm:sup}}]
(i). One can check that in the same way that Theorem~\ref{thm:sub}(ii) is proved using 
Theorem~\ref{thm:subMeanSquaredDisplacement}, one can use
Theorem~\ref{thm:supMeanSquaredDisplacement} to prove this part.
The details are left to the reader.\\
(ii). Fix $\epsilon<1-\theta_d(p)$.
Let $\rho = \rho(d,p,\epsilon) := \sqrt{\frac{1-\theta_d(p)+\epsilon}{2(1-\theta_d(p))}}\in(0,1)$.
By countable additivity, there exists $\kappa=\kappa(d,p,\epsilon)$ 
so that $\bfP_{\mathbb{Z}^{d},p}(|C(0)|\leq \kappa)\geq (1-\theta_d(p))\rho$.
For $n > \kappa$, we therefore have that
$\bfP_{\ndtorus,p}(|C(0)|\leq \kappa)\geq (1-\theta_d(p))\rho$.
Choose $C_{\ref{thm:sup}.2}=C_{\ref{thm:sup}.2}(d,p,\epsilon)$ sufficiently small so that
$e^{-2d\kappa C_{\ref{thm:sup}.2}}\geq \rho$. 

Now, let $C_t(0)$ be the cluster of the origin at time $t$.
For any $n$ larger than $\kappa$ and for any $\mu$,
conditioned on $\{|C_0(0)|\leq \kappa\}$, the conditional probability that no edges adjacent to 
$C_0(0)$ refresh during $[0,\frac{C_{\ref{thm:sup}.2}}{\mu}]$ is at least 
$e^{-2d\kappa C_{\ref{thm:sup}.2}}$ which was chosen larger than $\rho$.  
If $|C_0(0)|\leq \kappa$ and no edges adjacent to $C_0(0)$ refresh during 
$[0,\frac{C_{\ref{thm:sup}.2}}{\mu}]$, then it is necessarily the case that
$\dist(X_{\frac{C_{\ref{thm:sup}.2}}{\mu}},0)\le \kappa$. Hence
\[
\pr{\dist(X_{\frac{C_{\ref{thm:sup}.2}}{\mu}},0)\le \kappa|\delta_0\times \pi_p}
\ge (1-\theta_d(p))\rho^2=\frac{1-\theta_d(p)+\epsilon}{2} > \epsilon.
\]
Letting $E_n:=\{x\in\ndtorus:\dist(x,0)\le \kappa\}$, we therefore have
\[
\pr{X_{\frac{C_{\ref{thm:sup}.2}}{\mu}}\in E_n|\delta_0\times \pi_p} > \epsilon.
\]
On the other hand, it is clear that $u(E_n)$ goes to 0 as $n\to\infty$. 
This demonstrates (\ref{eq:SecondLowerBoundInSuper}) and completes the proof.
\end{proof}

We end this section by stating a proposition concerning the relaxation time analogous to 
Proposition \ref{prop:relaxation} which holds for all $p$. The proof of this proposition 
follows the proof of Proposition \ref{prop:relaxation} in a similar way to how the
proof of Theorem~\ref{thm:supMeanSquaredDisplacement} followed the proof of
Theorem~\ref{thm:subMeanSquaredDisplacement}. The details are left to the reader.

\begin{proposition}\label{prop:relaxationGeneral}
For any $d$, there exists 
$C_{\ref{prop:relaxationGeneral}}=C_{\ref{prop:relaxationGeneral}}(d)>0$
such that, for all $n$, $p$ and $\mu$, the relaxation time of the full system is at least
$C_{\ref{prop:relaxationGeneral}}n^2$.
\end{proposition}

\section{Proof of the mixing time upper bound in the subcritical case}
\label{sec:UpperBoundSub}

In this section, we prove Theorem \ref{thm:sub}(i).
This section is broken into four subsections. The first sets up the key technique of increasing
the state space, the second gives a sketch of the proof, the third provides
some percolation preliminaries and the fourth finally gives the proof.

\subsection{Increasing the state space in the general case}

In this subsection, we fix an arbitrary graph $G=(V,E)$ with constant degree
and parameters $p$ and $\mu$ and consider the resulting random
walk in dynamical percolation which we denote by $\{M_t\}_{t\ge 0}=\{(X_t,\eta_t)\}_{t\ge 0}$.
In order to obtain upper bounds on the mixing time, it will be useful to introduce
another Markov process which we denote by $\{\tilde{M}_t\}_{t\ge 0}=\{(X_t,\tilde{\eta}_t)\}_{t\ge 0}$
which will incorporate more information than $\{M_t\}_{t\ge 0}$;
the extra information will be the set of edges that the random walker has attempted to
cross since their last refresh time. The state space for this Markov process will be
\begin{equation}\label{eq:DefOfOmega}
\Omega:=\{(v,\tilde{\eta})\in V\times \{0,1,0^\fresh,1^\fresh\}^E:\tilde{\eta}(e)\in\{0,1\}\mbox{ for each $e$ adjacent to } v\}.
\end{equation}
If we identify $0^\fresh$ with $0$ and $1^\fresh$ with $1$, we want to
recover our process $\{M_t\}_{t\ge 0}$.
The idea of the possible extra $\fresh$ for the state of the edge $e$
is that this will indicate that the walker has not touched the endpoints of $e$ 
since $e$'s last refresh time. Hence, for such an edge $e$, 
whether there is a $1^\fresh$ or $0^\fresh$ at $e$ at that time is independent of everything else.

With the above in mind, it should be clear that we should define $\{\tilde{M}_t\}_{t\ge 0}$ 
as follows. An edge refreshes itself at rate $\mu$. Independent of everything else before 
the refresh time, the state of the edge after the refresh time will be $1^\fresh$ with 
probability $p$ and $0^\fresh$ with probability $1-p$ unless the edge is adjacent to the 
walker at that time. If the latter is the case, then the state of the edge after the refresh 
time will instead be $1$ with probability $p$ and $0$ with probability $1-p$. The random 
walker will as before choose at rate 1 a uniform neighbor (in the original graph) and move 
along that edge if the edge is in state 1 and not if the edge is in state 0. 
(Note that this edge can only be in state 1 or 0 since it is adjacent to the walker.) Finally, when
the random walker moves along an edge, the $\fresh$'s are removed from all edges which 
become adjacent to the walker. Clearly, dropping $\fresh$'s recovers the original process 
$\{M_t\}_{t\ge 0}$. We call an edge open if its state is $1^\fresh$ or $1$ and closed otherwise.

In order to exploit the $\fresh$-edges, we want that conditioned on (1) the position of the walker,
(2) the collection of $\fresh$-edges and (3) the states of the non-$\fresh$-edges,
we have no information concerning the states of the $\fresh$-edges. This is not necessarily true
for all starting distributions. We therefore restrict ourselves to a certain class of distributions. 
To define this, we first let
\[
\Pi:\{0,1,0^\fresh,1^\fresh\}^E\to\{0,1,\fresh\}^E
\]
be defined by identifying $1^\fresh$ and $0^\fresh$.

\begin{definition}\label{defn:GoodDistribution}
A probability measure on
$\Omega$ (as defined in (\ref{eq:DefOfOmega})) is called {\em good} if 
conditioned on $(v,\Pi(\eta))$, the conditional distribution of $\eta$ at the 
$\fresh$-edges is i.i.d.\ $1^\fresh$ with probability $p$ and $0^\fresh$ with probability $1-p$.
\end{definition}

Note that if $\mu$ is supported on $V\times\{0,1\}^E$, then $\mu$ is good. 

We will let $\{\mathcal{F}_t\}_{t\ge 0}$ be the natural filtration of $\sigma$-algebras 
generated by $\{M_t\}_{t\ge 0}$ which also keeps track of all the refresh times
and the attempted steps made by the walker. Note that
$\{\tilde{M}_t\}_{t\ge 0}$ is measurable with respect to this filtration.
Next, let $\{\mathcal{F}^\fresh_t\}_{t\ge 0}$ be the smaller filtration of $\sigma$-algebras
which is obtained when one does not distinguish $1^\fresh$ and $0^\fresh$ but is otherwise
the same. This filtration will be {\em critical} for our analysis.

A key property of good distributions, which also indicates the importance of the filtration
$\{\mathcal{F}^\fresh_t\}_{t\ge 0}$, is given in the following obvious lemma, whose proof is left to the reader. 

\begin{lemma}\label{lem:GoodDistributionsPreserved}
If the starting distribution for the Markov process $\{\tilde{M}_t\}_{t\ge 0}$ is good,
then, for all $s$, the conditional distribution of $\tilde{M}_s$ given $\mathcal{F}^\fresh_s$ is good,
as is the unconditional distribution of $\tilde{M}_s$. More generally, if $S$ is a 
$\{\mathcal{F}^\fresh_t\}_{t\ge 0}$ stopping time, then the conditional distribution of $\tilde{M}_S$ given 
$\mathcal{F}^\fresh_S$ is good, as is the unconditional distribution of $\tilde{M}_S$.
\end{lemma}

\subsection{Sketch of proof}
In order to make this argument more digestable, we explain here first the outline of the proof. 
Throughout this and the next
subsection, our processes of course depend on $d,p,n$ and $\mu$; however, we will drop these
in the notation throughout which will not cause any problems.
We start $\{\tilde{M}_t\}_{t\ge 0}$ with two initial configurations both in 
$\ndtorus\times\{0,1\}^{E(\ndtorus)}$; recall that these are necessarily good distributions.

We want to find a coupling of the two processes and a random time $T$ with mean of order at most 
$n^2/\mu$ so that after time $T$, the two configurations agree. Since $\{M_t\}_{t\ge 0}$ is obtained 
from $\{\tilde{M}_t\}_{t\ge 0}$ by dropping the $\fresh$'s, we will obtain our result.

This coupling will be achieved in three distinct stages. 

Stage 1. \\
In this first phase, we will 
run the processes independently until they simultaneously reach the set 
\begin{equation}\label{eq:DefRegeneration}
\regeneration := \{(x,\tilde{\eta}) \in \Omega \colon \tilde{\eta}(e)=0 \text{ for all $e$ adjacent to $x$ and } 
\tilde{\eta}(e)\in\{1^\fresh,0^\fresh\} \text{ for all other } e \}.
\end{equation}
Proposition \ref{prop:CouplingFirstPart} 
says that this will take at most order $\log n/\mu$ units of time. To prove this, one 
considers the set of edges
\[
A_s:=\{e:\tilde{\eta}_s(e)\in \{0,1\}\}\,\,\, (\mbox{i.e., the set of edges without a $\fresh$ at time $s$}).
\]
The hardest step is to show that on the appropriate time scale of order $1/\mu$,
the sets $A_s$ tend to decrease in size; this is the content of 
Proposition \ref{prop:DecreaseOfAs}, which relies on comparisons with subcritical percolation.
The fact that $A_s$ tends to decrease is intuitive as follows. A fixed proportion of the
set $A_s$ will be refreshed
during an interval of order $1/\mu$ while the random walker (which is causing
$A_s$ to increase by encountering new edges) is somewhat confined even on this time scale
since we are in a subcritical setting. Next 
Lemma \ref{lem:FixedChanceOfHittingHome} will tell us that once 
$A_s$ is relatively small, then the process will enter $\regeneration$
within a time interval of order $1/\mu$ with a fixed positive 
probability. Proposition \ref{prop:DecreaseOfAs} and Lemma \ref{lem:FixedChanceOfHittingHome} 
will allow us to prove Proposition \ref{prop:CouplingFirstPart}.

Stage 2. \\
At the start of the second stage, the two distributions are the same up to a translation $\sigma$.
At this point, we look at excursions from $\regeneration$ at discrete times on the time scale $1/\mu$.
Proposition \ref{prop:DecreaseOfAs} and Lemma \ref{lem:FixedChanceOfHittingHome} will now be used
again in conjunction with Proposition \ref{prop:GoodProcessesDie} 
to show that the number of steps in such an excursion is of order 1 which means order 
$1/\mu$ in real time; this is stated in Theorem \ref{thm:RenewalsHaveMeanOne}.
The joint distribution of the number of steps in an excursion and the increment of the walker during this 
excursion is complicated but it has a component of a fixed size where the excursion is one step and the 
increment is a lazy simple random walk. Coupling lazy simple random walk on
$\dntorus$ takes on order $n^2$ steps and so we can couple two copies of our process by
having them do the exact same thing off of this component of the distribution
and doing usual lazy simple random walk coupling on this component.
Since this component has a fixed probability, this coupling will couple in order $n^2$ excursions
and hence in order $n^2/\mu$ time. 

Stage 3. \\
After this, we can couple the full systems by a color switch.

We carry this all out in detail at the end of this section.

\subsection{Some percolation preliminaries}

In this subsection, we gather a number of results concerning percolation.

\begin{theorem}\label{thm:exp.volume}
For any $d\ge 1$ and $\alpha\in (0,p_\critical(\mathbb{Z}^{d}))$, there exists 
$C_{\ref{thm:exp.volume}}=C_{\ref{thm:exp.volume}}(d,\alpha)>0$ so that for all $r\ge 2$,
\[
\bfP_{\mathbb{Z}^{d},\alpha}(|C(0)|\ge r)\le e^{-C_{\ref{thm:exp.volume}}r}.
\]
The previous line holds with $\mathbb{Z}^{d}$ replaced by $\dntorus$.
\end{theorem}

\begin{proof}
This is Theorem~6.75 in \cite{Grimmett} in the case of $\mathbb{Z}^{d}$. Next, Theorem 1 in \cite{BS96} states
that if one has a covering map from a graph $G$ to a graph $H$, then the size of a vertex component in $H$ 
is stochastically dominated by the size of the corresponding vertex component in $G$. 
(This is stated for site percolation but site percolation is more general than bond percolation.)
Since we have a covering map from $\mathbb{Z}^{d}$ to $\dntorus$, we obtain the result 
for $\dntorus$ from the result for $\mathbb{Z}^{d}$.
\end{proof}

We collect here some graph theoretic definitions that we will need.

\begin{definition}\label{defn:EdgeSet}
If  $V'\subseteq V$, then $E(V')$ will denote the set
$\{e\in E: \exists v\in V' \mbox{ with } v\in e\}$.
(It is not required that both endpoints of $e$ are in $V'$.) 
\end{definition}

\begin{definition}\label{defn:VertexSet}
If $E'\subseteq E$, then $V(E')$ will denote the union of the endpoints of the edges in $E'$.
\end{definition}

\begin{definition}\label{defn:neighborhood}
If  $V'\subseteq V$, then
$\mathcal{N}_k(V'):=\{x\in V: \exists v\in V' \mbox{ with } \dist(x,v)\le k\}$ will be 
called the $k$-neighborhood of $V'$.
\end{definition}

\begin{definition}\label{defn:cluster}
If $V'\subseteq V$, then $E\backslash V'$ is defined to be those 
edges in $E$ which have at least one endpoint not in $V'$. 
\end{definition}

Given a set of vertices $F$ of $\mathbb{Z}^{d}$ or $\dntorus$ and a bond configuration $\eta$, let
$F^\eta$ to be the set of vertices reachable from $F$ using open edges in $\eta$. 
If $F$ is a set of vertices, then the configuration $\eta$ might only be specified for edges
in $E\backslash F$ but note that this has no consequence for the definition of $F^\eta$.
For a set of vertices $F$, we also let $F^\alpha$ be the random set obtained by 
choosing $\eta\subseteq E\backslash F$ according to $\pi_\alpha$ and then taking $F^\eta$.
We let $F^{\alpha,1}:=F^\alpha$ and we also define inductively, for $L\ge 2$,
$F^{\alpha,L}:=(F^{\alpha,L-1})^\alpha$. It is implicitly assumed here that we use independent 
randomness in each iteration.

\begin{theorem}\label{thm:Literations}
Fix $d\ge 1$ and $\alpha\in (0,p_\critical(\mathbb{Z}^{d}))$. Then for all $L$, there exists 
$C_{\ref{thm:Literations}}(L)=C_{\ref{thm:Literations}}(d,\alpha,L)$ 
so that for all finite $F\subseteq \mathbb{Z}^{d}$ and for all $\ell\ge 1$,
\[
\bfP(F^{\alpha,L}\not\subseteq \mathcal{N}_\kappa(F))\le \,\,
\frac{L}{2^{\frac{\ell}{\log \ell}}}
\]
where $\kappa:=\ell C_{\ref{thm:Literations}}(L)\log(|F|\vee 2)$.
In addition, for the case $L=1$, the $\log \ell$ term can be removed.
Finally, the same result holds for $\dntorus$ as well.
\end{theorem}

\begin{proof}
The following proof works for both $\mathbb{Z}^{d}$ and $\dntorus$.
We prove this by induction on $L$. The case $L=1$ without the $\log \ell$ term follows easily from 
Theorem \ref{thm:exp.radius} and is left to the reader. We now assume the result for $L=1$ 
(without the $\log \ell$ term) and for $L-1$ and prove it for $L$. It is elementary to check that
\[
\{F^{\alpha,L}\not\subseteq \mathcal{N}_\kappa(F)\}\subseteq
E_1\cup(E_2\cap E_3)
\]
where $E_1:=\{F^{\alpha,L-1}\not\subseteq \mathcal{N}_{\kappa_1}(F)\}$,
$E_2:=\{F^{\alpha,L-1}\subseteq \mathcal{N}_{\kappa_1}(F)\}$,
$E_3:=\{F^{\alpha,L}\not\subseteq \mathcal{N}_{\kappa-\kappa_1}(F^{\alpha,L-1})\}$ and 
$\kappa_1:=\ell C_{\ref{thm:Literations}}(L-1)\log(|F|\vee 2)$.

The probability of the first event is, by induction, at most 
$\frac{L-1}{2^{\frac{\ell}{\log \ell}}}$. Note next that when $E_2$ occurs, it is necessarily the case that
\[
|F^{\alpha,L-1}|\le |F|(2\ell C_{\ref{thm:Literations}}(L-1)\log(|F|\vee 2)+1)^d.
\]
The latter yields
\begin{equation}\label{eq:BoundAtTimeLminus1}
\log(|F^{\alpha,L-1}|\vee 2)\le 
\log(|F|)+d\log(2\ell C_{\ref{thm:Literations}}(L-1)\log(|F|\vee 2)+1).
\end{equation}
Now the neighborhood size arising in the event $E_3$ is
\[
\frac{\ell\log(|F|\vee 2) (C_{\ref{thm:Literations}}(L)-C_{\ref{thm:Literations}}(L-1))}
{C_{\ref{thm:Literations}}(1) \log(|F^{\alpha,L-1}|\vee 2)} \times  C_{\ref{thm:Literations}}(1) \log(|F^{\alpha,L-1}|\vee 2).
\]
By (\ref{eq:BoundAtTimeLminus1}), this first factor is at least 
\[
\frac{\ell\log(|F|\vee 2) (C_{\ref{thm:Literations}}(L)-C_{\ref{thm:Literations}}(L-1))}
{C_{\ref{thm:Literations}}(1) (\log(|F|)+d\log(2\ell C_{\ref{thm:Literations}}(L-1)\log(|F|\vee 2)+1))}. 
\]
It is easy to show that given $C_{\ref{thm:Literations}}(1)$
and $C_{\ref{thm:Literations}}(L-1)$, one can choose
$C_{\ref{thm:Literations}}(L)$ sufficiently large so that for all $F$ and for all $\ell$,
this is larger than $\frac{\ell}{\log \ell}$. It now follows from the $L=1$ case 
(where no $\log \ell$ term appears) that 
$\bfP(E_2\cap E_3)\le \frac{1}{2^{\frac{\ell}{\log \ell}}}$. 
Adding this to the first term yields the result.
\end{proof}

The previous theorem gave bounds on how far $F^\alpha$ (and its higher iterates) could be from $F$. 
The next proposition yields bounds on the size of $F^\alpha$ in terms of $F$. 
We will only need a bound on the mean which would then be easy to extend to higher iterates.

\begin{theorem}\label{thm:LiterationsVolume}
Fix $d\ge 1$ and $\alpha\in (0,p_\critical(\mathbb{Z}^{d}))$. Then there is a constant 
$C_{\ref{thm:LiterationsVolume}}=C_{\ref{thm:LiterationsVolume}}(d,\alpha)$ so that for all finite
$F\subseteq \mathbb{Z}^{d}$, one has
\[
\mathbf{E}{|F^\alpha|}  \le C_{\ref{thm:LiterationsVolume}}|F|.
\]
The same result holds for $\dntorus$ as well.
\end{theorem}

\begin{proof}
The following proof works for both $\mathbb{Z}^{d}$ and $\dntorus$. Theorem \ref{thm:exp.radius}
immediately implies that $\mathbf{E}_{d,\alpha}{|C(0)|}<\infty$. Note now that
$F^\alpha\subseteq \cup_{x\in F}\, C(x)$ where $C(x)$ is the set of vertices that can be reached from $x$ 
using the $\alpha$-open edges in $E\backslash F$ and so
$|F^\alpha|\le \sum_{x\in F}|C(x)|$. This now gives the result with 
$C_{\ref{thm:LiterationsVolume}}(d,\alpha)=\mathbf{E}_{d,\alpha}[|C(0)|]$.
\end{proof}

\subsection{Details of the proof}
We now fix $d$ and $p\in (0,p_\critical(\mathbb{Z}^{d}))$ for the rest of the argument.
We next choose $\epsilon=\epsilon(d,p)$ so that 
\[
\epsilon< \frac{p_c-p}{4}
\]
and we may assume that $\epsilon$ is an inverse integer. 
Since the probability that an edge is refreshed during an interval of length 
$\epsilon/\mu$ is $1-e^{-\epsilon}<\epsilon$, it follows that if, conditioned on 
${\mathcal{F}}^\fresh_s$,
an edge is open at some fixed time $t\ge s$ with conditional probability at most $\frac{p_c+p}{2}$, then 
the conditional probability given ${\mathcal{F}}^\fresh_s$ that $e$ is open some time during 
$[t,t+\frac{\epsilon}{\mu}]$ is at most 
\[
\frac{p_c+p}{2}+\epsilon <\frac{3}{4}p_c+\frac{1}{4}p.
\]
In particular, the unconditional probability that $e$ is open some time 
during $[t,t+\frac{\epsilon}{\mu}]$ is at most $\frac{3}{4}p_c+\frac{1}{4}p$. For notational convenience, we let
\[
p':=\frac{3}{4}p_c+\frac{1}{4}p.
\]

Proposition \ref{prop:DecreaseOfAs} is a crucial result showing  $A_s$ decreases by a fixed amount on 
an appropriate time scale. Before doing this, we need the following lemma 
which yields an apriori bound on the growth rate of $|A_s|$.

\begin{lemma}\label{lem:sub.BoundedGrowth}
There exists a constant $C_{\ref{lem:sub.BoundedGrowth}}=C_{\ref{lem:sub.BoundedGrowth}}(d,p)$
so that  for all $n$, $\mu$, $k\in \mathbb{N}$ and $s$, if we consider 
$\{\tilde{M}_t\}_{t\ge 0}$ with an arbitrary good initial distribution, then
\begin{equation}\label{eq:sub.BoundedGrowth}
\E{|A_{s+\frac{k\epsilon}{\mu}}|\mid{\mathcal{F}}^\fresh_s}\le 
C_{\ref{lem:sub.BoundedGrowth}}^k \, |A_s| \,\, { \rm a.s.}
\end{equation}
Moreover $C_{\ref{lem:sub.BoundedGrowth}}$ can be taken to be
\[
4d C_{\ref{thm:LiterationsVolume}}(d,p')
\]
where the latter constant comes from Theorem \ref{thm:LiterationsVolume}.
\end{lemma}

\begin{proof}
We prove this by induction on $k$. The main step is $k=1$. Define a random configuration
$\overline{\eta}$ of edges of $E\backslash A_s$ consisting of those edges which are open some time during
$[s,s+\epsilon/\mu]$. Note that, by Lemma \ref{lem:GoodDistributionsPreserved} and the way $\epsilon$ was 
chosen, for all $\mu$ and $n$, conditioned on ${\mathcal{F}}^\fresh_s$,
$\overline{\eta}$ is i.i.d.\ with density at most $p'$. The key observation to make is that 
\[
A_{s+\frac{\epsilon}{\mu}}\subseteq E(V(A_s)^{\overline{\eta}}) \,\, \mbox{ (see 
Definitions \ref{defn:EdgeSet} and \ref{defn:VertexSet}) }.
\]
It follows from Theorem \ref{thm:LiterationsVolume} that
\[
\E{|A_{s+\frac{\epsilon}{\mu}}|}\le C_{\ref{lem:sub.BoundedGrowth}}\, |A_s| 
\]
where $C_{\ref{lem:sub.BoundedGrowth}}=4d C_{\ref{thm:LiterationsVolume}}(d,p')$.
For $k\ge 2$, we have
\[
\E{|A_{s+\frac{k\epsilon}{\mu}}|\mid{\mathcal{F}}^\fresh_s}=
\E{\E{|A_{s+\frac{k\epsilon}{\mu}}|\mid{\mathcal{F}}^\fresh_{s+\frac{(k-1)\epsilon}{\mu}}}\mid{\mathcal{F}}^\fresh_s}\le
\]
\[
\E{C_{\ref{lem:sub.BoundedGrowth}}\,|A_{s+\frac{(k-1)\epsilon}{\mu}}|\mid{\mathcal{F}}^\fresh_s}\le 
C_{\ref{lem:sub.BoundedGrowth}}^k \, |A_s|
\]
where the first inequality follows from the $k=1$ case already proved and the last inequality follows by induction.
\end{proof}

We now move to

\begin{proposition}\label{prop:DecreaseOfAs}
There exist positive constants
$C_{\ref{prop:DecreaseOfAs}.1}=C_{\ref{prop:DecreaseOfAs}.1}(d,p)$ and
$C_{\ref{prop:DecreaseOfAs}.2}=C_{\ref{prop:DecreaseOfAs}.2}(d,p)$
so that for all $n$ and for all $\mu$, if we consider 
$\{\tilde{M}_t\}_{t\ge 0}$ started with an arbitrary good initial distribution, 
then we have that
\begin{equation}\label{eq:DecreaseOfAs}
\E{|A_{s+\frac{C_{\ref{prop:DecreaseOfAs}.1}}{\mu}}|\mid{\mathcal{F}}^\fresh_s}
\le \frac{|A_s|}{4}+ C_{\ref{prop:DecreaseOfAs}.2}\log |A_s|.
\end{equation}
\end{proposition}

\begin{proof}
Given ${\mathcal{F}}^\fresh_s$, the conditional probability that an edge $e$ is not refreshed during
$[s,s+\frac{k_1}{\mu}]$ or is refreshed and is open at time $s+\frac{k_1}{\mu}$ is
$e^{-k_1}+(1-e^{-k_1})p$. Choose an integer $k_1=k_1(d,p)$ sufficiently large so that 
\[
e^{-k_1}+(1-e^{-k_1})p\le \frac{p_c+p}{2}.
\]
Given a sufficiently large integer $k_2=k_2(d,p)$ to be chosen later, we let 
$t_1:=s+\frac{k_1}{\mu}$ and $t_2:=t_1+\frac{k_2}{\mu}$. 

Denoting the range of the random walker during the time interval $[u,v]$ by
$\mathcal{R}[u,v]$, a key observation is that 
\[
|A_{s+\frac{k_1+k_2}{\mu}}|\le  Q + |E(\mathcal{R}[t_1,t_2])| \,\, \mbox{ (see Definition \ref{defn:EdgeSet}) }
\]
where $Q$ is the number of edges in $A_{t_1}$ which are not refreshed during $[t_1,t_2]$.
The proof would be completed (by taking $C_{\ref{prop:DecreaseOfAs}.1}$ to be $k_1+k_2$)
if we can choose $k_2=k_2(d,p)$ 
and $C_{\ref{prop:DecreaseOfAs}.2}=C_{\ref{prop:DecreaseOfAs}.2}(d,p)$ so that for all $n$ and $\mu$,
\begin{equation}\label{eq:sub.BoundOnSecondTerm}
\E{Q\mid {\mathcal{F}}^\fresh_s}\le \frac{|A_s|}{4}
\end{equation}
and
\begin{equation}\label{eq:sub.BoundOnFirstTerm}
\E{|E(\mathcal{R}[t_1,t_2])|\mid{\mathcal{F}}^\fresh_s}
\le  C_{\ref{prop:DecreaseOfAs}.2}\log |A_s|.
\end{equation}

We start by finding $k_2$ so that (\ref{eq:sub.BoundOnSecondTerm}) holds for all $n$ and $\mu$.
For this, we simply choose $k_2$ so that
\begin{equation}\label{eq:sub.k2Choice}
e^{-k_2}C_{\ref{lem:sub.BoundedGrowth}}^{\frac{k_1}{\epsilon}}\le \frac{1}{4}
\end{equation}
where $C_{\ref{lem:sub.BoundedGrowth}}$ comes from Lemma \ref{lem:sub.BoundedGrowth} and note that
$k_2$ only depends on $d$ and $p$.

Clearly $\E{Q\mid{\mathcal{F}}^\fresh_{t_1}}= e^{-k_2} |A_{t_1}|$ which implies that
\[
\E{Q\mid{\mathcal{F}}^\fresh_s}=
\E{\E{Q \mid{\mathcal{F}}^\fresh_{t_1}}\mid{{\mathcal{F}}^\fresh_s}}=
\E{e^{-k_2} |A_{t_1}|\mid{\mathcal{F}}^\fresh_s}.
\]
Since $t_1=s+\frac{k_1}{\epsilon}\frac{\epsilon}{\mu}$ (with $\epsilon$ being an inverse integer), 
Lemma \ref{lem:sub.BoundedGrowth} and (\ref{eq:sub.k2Choice}) imply that this last expression is at most
$\frac{|A_s|}{4}$, demonstrating (\ref{eq:sub.BoundOnSecondTerm}).

With $k_1$ and $k_2$ now chosen, we want to find
$C_{\ref{prop:DecreaseOfAs}.2}=C_{\ref{prop:DecreaseOfAs}.2}(d,p)$ so that for all $n$ and $\mu$,
(\ref{eq:sub.BoundOnFirstTerm}) holds. Since
$|E(\mathcal{R}[t_1,t_2])|\le 2d|\mathcal{R}[t_1,t_2]|$, it suffices to prove
such a bound for $|\mathcal{R}[t_1,t_2]|$ instead.

We first observe that by the way $k_1$ was chosen, conditioned on ${\mathcal{F}}^\fresh_s$, 
the probability that an arbitrary edge $e$ is open
at some fixed time $t\ge t_1$ is at most $\frac{p_c+p}{2}$ and hence by our choice of $\epsilon$, we have that
for any interval $I:=[y,y+\frac{\epsilon}{\mu}]\subseteq [t_1,t_2]$,
the conditional probability given ${\mathcal{F}}^\fresh_s$
that $e$ is open some time during $[y,y+\frac{\epsilon}{\mu}]$  is at most $p'$.
Note that conditioned on ${\mathcal{F}}^\fresh_s$, 
the evolution of the states of the different edges after time $s$ are independent although
they are not identically distributed.

We next partition $[t_1,t_2]$ into $D=D(d,p)$ disjoint intervals of length 
$\frac{\epsilon}{\mu}$. Note importantly that $D$ does not depend on $n$ or $\mu$.
It now suffices to show that
if $I:=[y,y+\frac{\epsilon}{\mu}]\subseteq [t_1,t_2]$ with $y=t_1+\frac{\ell\epsilon}{\mu}$ and
$\ell\in \mathbb{N}$, then
\begin{equation}\label{eq:sub.BoundOnFirstTermSimplified}
\E{|\mathcal{R}[I]|\mid{\mathcal{F}}^\fresh_s}
\le C\log |A_s| \,\, { \rm a.s.}
\end{equation}
for some $C$ depending only on $d$ and $p$. To do this, it suffices to show
that
\[
\|\sum_{\ell\ge 1} \Pruu{|\mathcal{R}[I]|\ge \ell \frac{4d}{C_{\ref{thm:exp.volume}}(d,p')}
\log |V(A_s)|\mid{\mathcal{F}}^\fresh_s}\|_{\infty} <\infty
\]
where $C_{\ref{thm:exp.volume}}(d,p')$ comes from Theorem \ref{thm:exp.volume}.
Let $\overline{\eta}$ be the set of edges which are open some time during $I$.
By our choice of $\epsilon$, conditioned on ${\mathcal{F}}^\fresh_s$,
$\overline{\eta}$ is stochastically dominated by an i.i.d.\ process with density $p'$.
Since $\mathcal{R}[I]$ is necessarily contained inside of a $\overline{\eta}$-cluster, we have
\begin{equation}\label{eq:ImportantBreakup}
\Pruu{|\mathcal{R}[I]|\ge \ell \frac{4d}{C_{\ref{thm:exp.volume}}(d,p')}
\log |V(A_s)| \mid{\mathcal{F}}^\fresh_s}\le \Pruu{\dist(X_y,X_s)\ge 
 \ell\, \Xi \mid{\mathcal{F}}^\fresh_s}+
\end{equation}
\[
\Pruu{\overline{\eta} \mbox{ contains a cluster of size } \ge \ell
\frac{4d}{C_{\ref{thm:exp.volume}}(d,p')} \log |V(A_s)| \mbox{ intersecting } 
B(X_s, \ell \, \Xi)\mid{\mathcal{F}}^\fresh_s}
\]
where 
\[
\Xi:=4d C_{\ref{thm:Literations}}(d,p',D) |V(A_s)| \log |V(A_s)|,
\]
where $C_{\ref{thm:Literations}}(d,p',D)$ comes
from Theorem \ref{thm:Literations} and where, as before,
$B(v,r)$ is the set of vertices within $\dist$-distance $r$ of $v$.

It is easy to check that Theorem \ref{thm:exp.volume} together with a union bound implies that
the second terms are summable over $\ell$ uniform in the conditioning, i.e.
\[
\|\sum_{\ell\ge 1} 
\Pruu{\overline{\eta} \mbox{ contains a cluster of size } \ge \ell
\frac{4d}{C_{\ref{thm:exp.volume}}(d,p')} \log |V(A_s)| \mbox{ intersecting } 
B(X_s, \ell \, \Xi)\mid{\mathcal{F}}^\fresh_s}\|_{\infty} <\infty.
\]
To deal with the first terms, we partition
the interval $[s,y]$ into successive intervals $J_1,J_2,\ldots,J_{L}$ of lengths $\frac{\epsilon}{\mu}$
where $L\le D$. Let $\overline{\eta}_i$ be the set of edges which are open some time during $J_i$.
The key observation is (see Definition \ref{defn:neighborhood}) that for each $w\ge |V(A_s)|\log |V(A_s)|$
\begin{equation}\label{eq:DistantGivesOutsideNeighborhood}
\{\dist(X_y,X_s)> 4dw\}\subseteq 
\{V(A_s)^{\overline{\eta_1},\ldots,\overline{\eta_{L}}}\not\subseteq \mathcal{N}_{\frac{w}{|V(A_s)|}}(V(A_s))\}.
\end{equation}
To see this, one first makes the important observation that the random walk path between times $s$ and $y$
is contained in $V(A_s)^{\overline{\eta_1},\ldots,\overline{\eta_{L}}}$ and a geometric argument shows that 
$\{\dist(X_y,X_s)> 4dw\}$ implies that $\mathcal{R}[s,y]$ cannot be contained
in $\mathcal{N}_{\frac{w}{|V(A_s)|}}(V(A_s))$.

We claim that conditioned on ${\mathcal{F}}^\fresh_s$, the set 
\begin{equation}\label{eq:StochDom}
V(A_s)^{\overline{\eta}_1,\ldots,\overline{\eta}_L} \mbox{ is stochastically dominated by } V(A_s)^{p',L}. 
\end{equation}
To see this, one first notes that since we are assuming a good initial distribution,
when we condition on ${\mathcal{F}}^\fresh_s$, the edges not in $A_s$, which are the only relevant
edges in the construction of $V(A_s)^{\overline{\eta}_1,\ldots,\overline{\eta}_L}$, are in stationarity
and hence, conditioned on ${\mathcal{F}}^\fresh_s$, each
$\overline{\eta}_i$, off of the edge set $A_s$,
is stochastically dominated by an i.i.d.\ process with density $p'$. Second, one notes that
when one further conditions on the sets 
$V(A_s)^{\overline{\eta_1}},\ldots,V(A_s)^{\overline{\eta_1},\ldots,\overline{\eta_{j}}}$, 
this can only stochastically decrease the edges of $\overline{\eta_{j+1}}$ which are relevant at that 
point. Hence we obtain (\ref{eq:StochDom}). This stochastic domination implies that 
\begin{equation}\label{eq:StochComparison}
\Pruu{V(A_s)^{\overline{\eta_1},\ldots,\overline{\eta_{L}}}\not\subseteq \mathcal{N}_{\frac{w}{|V(A_s)|}}(V(A_s))}\le
\Pruu{V(A_s)^{p',L}\not\subseteq \mathcal{N}_{\frac{w}{|V(A_s)|}}(V(A_s))}.
\end{equation}
Letting $w:= \ell C_{\ref{thm:Literations}}(d,p',D)|V(A_s)|\log |V(A_s)|$,
(we may assume $C_{\ref{thm:Literations}}(d,p',D)\ge 1$)
(\ref{eq:DistantGivesOutsideNeighborhood}), (\ref{eq:StochComparison}) and Theorem
\ref{thm:Literations} now imply that the first terms of (\ref{eq:ImportantBreakup})
are also summable over $\ell$ uniform in the conditioning, i.e.,
\[
\|\sum_{\ell\ge 1} 
\Pruu{\dist(X_y,X_s)\ge  \ell\, \Xi \mid{\mathcal{F}}^\fresh_s}\|_{\infty} <\infty.
\]
\end{proof}

\begin{remark}{\rm If one wants, one could avoid the use of Theorem \ref{thm:exp.volume} and thereby 
the need for Theorem 1 from \cite{BS96}. One could do this by modifying the above proof using only 
Theorem \ref{thm:exp.radius} obtaining (\ref{eq:DecreaseOfAs}) with the $\log$ term replaced by a 
power of $\log$ which would suffice for the rest of the proof.} 
\end{remark}

We are now going to look at our process at integer multiples of 
$\frac{C_{\ref{prop:DecreaseOfAs}.1}}{\mu}$ where $C_{\ref{prop:DecreaseOfAs}.1}$ comes from Proposition
\ref{prop:DecreaseOfAs}. We therefore let, for integer $k\ge 0$,
\begin{equation}\label{eq:DefinitionDBG}
D_k:=(X_{\frac{kC_{\ref{prop:DecreaseOfAs}.1}}{\mu}}, \tilde{\eta}_{\frac{kC_{\ref{prop:DecreaseOfAs}.1}}{\mu}}), \,\, B_k:=A_{\frac{kC_{\ref{prop:DecreaseOfAs}.1}}{\mu}}
\mbox{ and } \mathcal{G}_k:=\mathcal{F}^\fresh_{\frac{kC_{\ref{prop:DecreaseOfAs}.1}}{\mu}}.
\end{equation}
We will use the set $\regeneration$ (see (\ref{eq:DefRegeneration})) as a 
regenerative set. We therefore 
define the set of regeneration times along our subsequence of times by
\begin{equation}\label{eq:DefinitionI}
\mathcal{I}:=\{k\ge 0: D_k\in \regeneration\}.
\end{equation}

Our next lemma says that when we are not far away from 
$\regeneration$, then we have a good chance of entering it.

In the proof of this lemma, if  $E'\subseteq E$, then we let
$\partial(E')$ denote the set $\{e\not\in E': \exists e'\in E' \mbox{ with $e$ and $e'$ adjacent} \}$;
this will be called the boundary of $E'$.

\begin{lemma}\label{lem:FixedChanceOfHittingHome}
For all $R$, there exists $\alpha_{\ref{lem:FixedChanceOfHittingHome}}=
\alpha_{\ref{lem:FixedChanceOfHittingHome}}(d,p,R)>0$ so that for all $n$ and for all
$\mu$, if we consider 
$\{\tilde{M}_t\}_{t\ge 0}$ with an arbitrary good initial distribution, then
\[
\Pruu{j+1\in \mathcal{I} \mid \mathcal{G}_j} \ge 
\alpha_{\ref{lem:FixedChanceOfHittingHome}}
\]
on the event $|B_j|\le R$.
\end{lemma}

\begin{proof}
First, let $E_1$ be the event that all edges of the boundary of $B_j$, $\partial(B_j)$
are closed at time $\frac{jC_{\ref{prop:DecreaseOfAs}.1}}{\mu}$, 
$E_2$ be the event that no edge of $\partial(B_j)$ refreshes during
$[\frac{C_{\ref{prop:DecreaseOfAs}.1}}{\mu}j,\frac{C_{\ref{prop:DecreaseOfAs}.1}}{\mu}(j+\frac{1}{2})]$
and $E_3$ be the event that all edges in $B_j$ are refreshed closed during
$[\frac{C_{\ref{prop:DecreaseOfAs}.1}}{\mu}j,\frac{C_{\ref{prop:DecreaseOfAs}.1}}{\mu}(j+\frac{1}{2})]$.
Observe that if $E_1\cap E_2\cap E_3$ occurs, then the edges next to the walker
necessarily are closed and belong to either $B_j$ or $\partial(B_j)$.
Next, let $E_4$ be the event that the edges adjacent to 
$X_{\frac{C_{\ref{prop:DecreaseOfAs}.1}}{\mu}(j+\frac{1}{2})}$
do not refresh during
$[\frac{C_{\ref{prop:DecreaseOfAs}.1}}{\mu}(j+\frac{1}{2}),\frac{C_{\ref{prop:DecreaseOfAs}.1}}{\mu}(j+1)]$
and $E_5$ be the event that all the edges of $B_j$ and $\partial(B_j)$
except for the edges adjacent to 
$X_{\frac{C_{\ref{prop:DecreaseOfAs}.1}}{\mu}(j+\frac{1}{2})}$
refresh during 
$[\frac{C_{\ref{prop:DecreaseOfAs}.1}}{\mu}(j+\frac{1}{2}),\frac{C_{\ref{prop:DecreaseOfAs}.1}}{\mu}(j+1)]$.
It is elementary to check, using the fact that we started with a good initial distribution, that 
$\Pruu{\cap_{i=1}^5 E_i\mid \mathcal{G}_j}$
is bounded away from 0 on the event $|B_j|\le R$ uniformly in $n$ and $\mu$
and the conditioning and that 
$\cap_{i=1}^5 E_i\subseteq \{j+1\in \mathcal{I}\}$. This completes the proof.
\end{proof}

Our next proposition tells us that we can do the first stage of the coupling described in the sketch;
namely, two copies of our system will enter $\regeneration$ within $\log n$ many steps and hence within
$\log n/\mu$ units of time. 

\begin{proposition}\label{prop:CouplingFirstPart}
There exists $C_{\ref{prop:CouplingFirstPart}}=C_{\ref{prop:CouplingFirstPart}}(d,p)<\infty$ 
so that for all $n$ and for all $\mu$, if 
$\{D^1_k\}_{k\ge 0}$ and $\{D^2_k\}_{k\ge 0}$ are two independent copies of $\{D_k\}_{k\ge 0}$ each starting
from arbitrary initial configurations in $\ndtorus\times\{0,1\}^{E(\ndtorus)}$, and if we set
\[
T_1:=\min\{k\ge 1: D^1_k\in\regeneration \mbox{ and } D^2_k\in\regeneration\},
\]
then $\E{T_1}\le C_{\ref{prop:CouplingFirstPart}}\log n$.
\end{proposition}

\begin{proof}
Let $Z_k:=|B^1_k|+|B^2_k|$ using obvious notation.
By Proposition \ref{prop:DecreaseOfAs}, we obtain that for all $n$ and for all $\mu$, 
\[
\E{Z_{k+1}\mid {{\mathcal{G}}_k\times {\mathcal{G}}_k}}
\le \frac{Z_k}{4}+ 2C_{\ref{prop:DecreaseOfAs}.2}\log (Z_k).
\]
This immediately gives that there is a constant
$C_{\ref{prop:CouplingFirstPart}.1}=C_{\ref{prop:CouplingFirstPart}.1}(d,p)<\infty$ 
so that
\begin{equation}\label{eq:FactorDecrease}
\E{Z_{k+1}\mid {{\mathcal{G}}_k\times {\mathcal{G}}_k}}\le \frac{Z_k}{3}
\end{equation}
on the event $Z_k\ge C_{\ref{prop:CouplingFirstPart}.1}$. Now, noting that
$Z_0= 2dn^d$ (which is $\ge C_{\ref{prop:CouplingFirstPart}.1}$ for $n$ large),
(\ref{eq:FactorDecrease}) implies that $3^{k\wedge \tilde{T}}Z_{k\wedge \tilde{T}}$ is a 
${\mathcal{G}}_k\times {\mathcal{G}}_k$-supermartingale where
\[
\tilde{T}:=\min\{k: Z_k\le C_{\ref{prop:CouplingFirstPart}.1}\}.
\]
From the theory of stopping times for nonnegative supermartingales, we obtain the fact that
\[
\E{3^{\tilde{T}}Z_{\tilde{T}}}\le Z_0.
\]
Since the $Z_k$'s are always at least 1 and $Z_0= 2dn^d$, we infer that
\[
\E{3^{\tilde{T}}}\le 2dn^d
\]
which in turn implies, by Jensen's inequality, that
\[
\E{\tilde{T}}\le \log_3(2dn^d).
\]
Since both $|B^1_{\tilde{T}}|$ and $|B^2_{\tilde{T}}|$ are less than
$C_{\ref{prop:CouplingFirstPart}.1}$, using the fact that 
$\tilde{T}$ is a stopping time and the independence of the two processes,
we can conclude from Lemma \ref{lem:FixedChanceOfHittingHome}
that for all $n$ and $\mu$, the probability that both $D^1_{\tilde{T}+1}$ and $D^2_{\tilde{T}+1}$
are in $\regeneration$ is at least 
$\alpha_{\ref{lem:FixedChanceOfHittingHome}}(d,p,C_{\ref{prop:CouplingFirstPart}.1})^2$.
If this fails, we start again 
and wait on average another at most $\log_3(2n^d)$. After a geometric number of trials,
we are done. By Wald's Theorem, this proves the result.
\end{proof}

We now return to looking at just one copy of our system and study the excursions of $\{\tilde{M}_t\}_{t\ge 0}$
away from $\regeneration$. Assume now that $\tilde{M}_0$ has a good distribution supported on $\regeneration$. 
We let $\tau_0=0$ and, for $j\geq 1$, define 
\begin{equation}\label{eq:Tauj's}
\tau_j := \min\{i > \tau_{j-1} \colon i\in \mathcal{I}\}, \,\,\,
U_j:= D'_{\tau_j}-D'_{\tau_{j-1}}
\end{equation}
where $D'_k$ denotes the first coordinate of $D_k$.

Note that, by Lemma \ref{lem:GoodDistributionsPreserved}, for each $j$, the distribution of the process at time
$\frac{\tau_jC_{\ref{prop:DecreaseOfAs}.1}}{\mu}$ conditioned on $\mathcal{G}_{\tau_j}$ is good.

\begin{remark}\label{rem:DescripGoodMeasures} {\rm
It is easy to see that the set of good probability measures supported on $\regeneration$ can be 
described as follows; one chooses a vertex $v$ at random according to any distribution and then sets 
the edges adjacent to $v$ to be in state 0 and all other edges are (conditionally) independently 
chosen to be $1^\fresh$ with probability $p$ and $0^\fresh$ with probability $1-p$.
}\end{remark}

The following lemma, whose proof is left to the reader, is clear.
\begin{lemma}\label{lem:RegerationTimesIndep}
If our initial distribution is a good distribution supported on 
$\regeneration$, then \\
$\{(\tau_{j}-\tau_{j-1},U_j)\}_{j\ge 1}$ are i.i.d.\ and moreover, for each $j\ge 1$,
given $\mathcal{G}_{\tau_{j-1}}$, the conditional distribution of 
$(\tau_{j}-\tau_{j-1},U_j)$ is the same as $(\tau_{1},U_1)$.
\end{lemma}

\begin{remark}{\rm
Of course $\tau_{j}-\tau_{j-1}$ and $U_j$ are not independent of each other.
}\end{remark}

The next theorem tells us that the number of steps in one of our excursions away from 
$\regeneration$ is of order 1. 

\begin{theorem}\label{thm:RenewalsHaveMeanOne}
There exists 
$C_{\ref{thm:RenewalsHaveMeanOne}}=C_{\ref{thm:RenewalsHaveMeanOne}}(d,p)< \infty$
such that, for all $n$ and for all $\mu$, if we start with a good initial distribution
supported on $\regeneration$, then
\[
\E{\tau_1}\le C_{\ref{thm:RenewalsHaveMeanOne}}.
\]
\end{theorem}

\begin{proof}
Let $Y_k:=1+B'_k+L_k$ where $B'_k$ is the number of edges not adjacent to the walker which are
in state 0 or 1 at time $\frac{kC_{\ref{prop:DecreaseOfAs}.1}}{\mu}$ and where $L_k$ is
the number of edges adjacent to the walker which are in state $1$ at time 
$\frac{kC_{\ref{prop:DecreaseOfAs}.1}}{\mu}$. Note $Y_k=1$ if and only if $D_k\in \regeneration$;
hence $Y_0=1$ and $\tau_1$ corresponds to the first return of the $Y$ process
to 1. We will apply Proposition \ref{prop:GoodProcessesDie} with 
$\mathcal{F}_i$ there being $\mathcal{G}_i$. Property (1) trivially holds.
Proposition \ref{prop:DecreaseOfAs}
easily implies that for $\delta=1/3$ and for $\alpha$ sufficiently large,
but only depending on $d$ and $p$, Property (2) holds for all $n$ and $\mu$.
Next, with $\alpha$ now fixed, Lemma \ref{lem:sub.BoundedGrowth}
implies that there exists $\gamma$ sufficiently large,
but only depending on $d$ and $p$, so that Property (4) holds for all $n$ and $\mu$.
Finally, since $\alpha$ is now fixed,
Lemma \ref{lem:FixedChanceOfHittingHome} guarantees that property (3) holds for all 
$n$ and $\mu$ for some positive $\epsilon$ also only depending on $d$ and $p$.
Proposition \ref{prop:GoodProcessesDie} now yields the result.
\end{proof}

We now finally have all of the ingredients to give the 

\begin{proof}[{\bf Proof of Theorem \ref{thm:sub}(ii)}]
Let $\{\tilde{M}^1_t\}_{t\ge 0}$ and $\{\tilde{M}^2_t\}_{t\ge 0}$ denote
two copies of our process $\{\tilde{M}_t\}_{t\ge 0}$, each starting from an 
arbitrary configuration in $\ndtorus\times\{0,1\}^{E(\ndtorus)}$. We will find a coupling
$(\{\tilde{M}^1_t\}_{t\ge 0},\{\tilde{M}^2_t\}_{t\ge 0},T)$ 
of our two processes and a nonnegative random variable $T$ so that \\
(1) $\tilde{M}^1_t=\tilde{M}^2_t$ for all $t\ge T$ and \\
(2) $\E{T}\le O_{d,p}(1)\frac{n^2}{\mu}$.

From here, it is standard that this gives a bound on the mixing time as follows. If $t\ge 4\E{T}$, then
\[
\tv{\calL(M^1_t)}{\calL(M^2_t)} 
\le \Pruu{M^1_t\neq M^2_t}\le \Pruu{\tilde{M}^1_t\neq\tilde{M}^2_t}\le 
\Pruu{T>t}\le \frac{1}{4}
\]
by Markov's inequality. As outlined earlier, we do this coupling in two separate stages.

Stage 1. \\
In this first stage we run the two processes independently
until both processes simultaneously hit $\regeneration$ at some time $T_1$ of the form 
$\frac{kC_{\ref{prop:DecreaseOfAs}.1}}{\mu}$, $k\in \mathbb{N}$.
By Proposition \ref{prop:CouplingFirstPart},
this first stage will take in expectation at most $O_{d,p}(1)\frac{\log n}{\mu}$ time.
Let $\mathcal{F}^{1,\fresh}_{T_1}$ be the $\sigma$-algebra generated by 
$\{\tilde{M}^1_t\}_{0\le t\le T_1}$ including all the refresh times in $[0, T_1]$
but where one does not distinguish $1^\fresh$ and $0^\fresh$.
Let $\mathcal{F}^{2,\fresh}_{T_1}$ be defined analogously.
(A trivial variant of) Lemma \ref{lem:GoodDistributionsPreserved} and the independence 
of the two processes imply that, conditioned on (1) $T_1$, (2) $\mathcal{F}^{1,\fresh}_{T_1}$ 
and (3) $\mathcal{F}^{2,\fresh}_{T_1}$,
we have that these two conditional distributions at time $T_1$ are good. 
Also, the two conditional distributions of the walker are trivially degenerate.
In view of Remark \ref{rem:DescripGoodMeasures}, these two conditional distributions
will then agree up to a translation $\sigma$ which translates one walker to the other. 
Moreover, the two processes at time $T_1$ are conditionally independent given 
(1), (2) and (3) above.

Stage 2. \\
We now condition on (1), (2) and (3) above. Viewing things from time $T_1$,
we are then in the setting of Lemma \ref{lem:RegerationTimesIndep} where 
we can view our two processes as being independent and having
good distributions supported on $\regeneration$ with the walker having a degenerate
distribution.

Before completing the second stage, we first recall for the reader how one couples two 
lazy simple random walks on $\dntorus$. By definition, a lazy simple random walk on 
$\dntorus$ with probability $1/2$ 
stays where it is and with probability $1/2$ moves to a neighbor chosen uniformly at random.
The usual coupling of two lazy simple random walks on $\dntorus$ is a 
Markov process on $\dntorus\times\dntorus$
defined as follows. One chooses one of the $d$ coordinates uniformly at random. 
If the two walkers agree in this coordinate, then they both stay still, 
both move ``right'' in this coordinate or both move ``left'' in this coordinate 
with respective probabilities $1/2$, $1/4$ and $1/4$.
If the two walkers disagree in this coordinate, then one of the walkers, chosen at random,
stays still while the other one moves ``right'' or ``left'' in this coordinate, each with
probability $1/2$. It is easy to check that this is a coupling of the two lazy simple random walks.
It is known (see Section 5.3 in \cite{LPW}) that the expected time until the two walkers meet is at most $O_{d}(1)n^2$.

Denote the joint distribution of $(\tau_{1},U_1)$ by $\nu_{d,p,n,\mu}$; this is a probability measure
on $\mathbb{N}\times\dntorus$. It is easy to show, along the same lines as the proof of 
Lemma \ref{lem:FixedChanceOfHittingHome}, that for any $d$ and $p$, there is a $\gamma=\gamma(d,p)>0$ 
so that for any $n$ and $\mu$,
\begin{equation}\label{eq:SRWcomponent}
\nu_{d,p,n,\mu}=\gamma (\delta_1\times\nu_{\rm LSRW})+(1-\gamma) m_{d,p,n,\mu}
\end{equation}
where $\delta_1\times\nu_{\rm LSRW}$ is the probability measure on $\mathbb{N}\times\dntorus$ where 
the first coordinate is always 1 and the second coordinate is a step of a lazy simple random walk on
$\dntorus$ and where $m_{d,p,n,\mu}$ is some probability measure on $\mathbb{N}\times\dntorus$.

We will first couple the random sequence $\{(\tau_{j}-\tau_{j-1},U_j)\}_{j\ge 1}$ for the two systems and
denote the coupled variables by 
$\{\left((\tau^1_{j}-\tau^1_{j-1},U^1_j),(\tau^2_{j}-\tau^2_{j-1},U^2_j)\right)\}_{j\ge 1}$.
We will do this in such a way that $\tau^1_j-\tau^1_{j-1}=\tau^2_j-\tau^2_{j-1}$ for all $j$.
Given $\{\left((\tau^1_{j}-\tau^1_{j-1},U^1_j),(\tau^2_{j}-\tau^2_{j-1},U^2_j)\right)\}_{1\le j\le K}$
so that $\tau^1_j-\tau^1_{j-1}=\tau^2_j-\tau^2_{j-1}$ for all $1\le j\le K$, we 
know where the two walkers are located at time 
\[
\frac{C_{\ref{prop:DecreaseOfAs}.1}}{\mu}\sum_{j=1}^K \left(\tau^1_j-\tau^1_{j-1}\right)  \,\,\, 
(=\frac{C_{\ref{prop:DecreaseOfAs}.1}}{\mu}\tau^1_K
= \frac{C_{\ref{prop:DecreaseOfAs}.1}}{\mu}\tau^2_K
= \frac{C_{\ref{prop:DecreaseOfAs}.1}}{\mu}\sum_{j=1}^K \left(\tau^2_j-\tau^2_{j-1}\right)).
\]
We now define 
$\left((\tau^1_{K+1}-\tau^1_{K},U^1_{K+1}),(\tau^2_{K+1}-\tau^2_{K},U^2_{K+1})\right)$ as follows.
With probability $1-\gamma$, one chooses an element from 
$\mathbb{N}\times\dntorus$ according to distribution $m_{d,p,n,\mu}$ and uses it for 
both systems. 
With probability $\gamma$, one takes the first coordinate to be 1 for both systems and
does the coupled lazy simple random walk (described above) for the second coordinates of the two systems. 
(This coupling is reminiscent of a coupling due to D. Ornstein.) Note that this is a coupling and with it,
we have that $\tau^1_j-\tau^1_{j-1}=\tau^2_j-\tau^2_{j-1}$ for all $j$ and 
$(\tau^1_j-\tau^1_{j-1},U^1_j)=(\tau^2_j-\tau^2_{j-1},U^2_j)$ for all large $j$ depending on $\omega$.

By (\ref{eq:SRWcomponent}), the fact that $\gamma$ is uniformly bounded away from 0,
the fact that the standard coupling of two lazy simple random walks 
on $\dntorus$ couples in expected time at most $O_{d}(1)n^2$ and Wald's Theorem,
we have that the expected number of steps, denoted by $N$, in the above coupling until
the two walkers meet is at most $O_{d,p}(1)n^2$. 
Using Theorem \ref{thm:RenewalsHaveMeanOne}, another application of Wald's Theorem
tells us that the expected value of
\[
T_2:= \frac{C_{\ref{prop:DecreaseOfAs}.1}}{\mu}\sum_{j=1}^N \left(\tau^1_j -\tau^1_{j-1}\right)
\]
is at most $O_{d,p}(1)\frac{n^2}{\mu}$. 
We now let
\[
T:= T_1+ T_2
\]
and note that we have $\E{T}\le O_{d,p}(1)\frac{n^2}{\mu}$. 

Stage 3. \\
By construction, we have that $\{T_1,\mathcal{F}^{1,\fresh}_{T_1}\}$ is
independent of the random variables $\{(\tau^1_{j}-\tau^1_{j-1},U^1_j)\}_{j\ge 1}$ from the second stage
and similarly, we have that $\{T_1,\mathcal{F}^{2,\fresh}_{T_1}\}$ is
independent of the random variables $\{(\tau^2_{j}-\tau^2_{j-1},U^2_j)\}_{j\ge 1}$ from the second stage.
It also follows by construction that the conditional distribution of
$\{\tilde{M}^1_t\}_{t\ge T}$ conditioned on 
$\{T_1,\mathcal{F}^{1,\fresh}_{T_1},\{(\tau^1_{j}-\tau^1_{j-1},U^1_j)\}_{j\ge 1},T\}$
is the same as the conditional distribution of $\{\tilde{M}^2_t\}_{t\ge 0}$ conditioned on 
$\{T_1,\mathcal{F}^{2,\fresh}_{T_1},\{(\tau^2_{j}-\tau^2_{j-1},U^1_j)\}_{j\ge 1},T\}$.
This implies that we have a coupling of the desired form and completes the proof.
\end{proof}

\section{Hitting time results} \label{sec:hitting}

This section is devoted to proving Theorem \ref{thm:hitting}. We will use a number of the results which
were derived in Section \ref{sec:UpperBoundSub}.
We first need some lemmas which might be of independent interest.

Our first lemma follows easily from the usual local central limit theorem for lazy simple random walk on 
$\mathbb{Z}^d$ and therefore no proof is given.
\begin{lemma}\label{lem:ReplacementOfLCLT}
Fix $d\ge 1$ and let $P_{n}^k$ be the $k$-step transition probability for lazy simple random walk on $\dntorus$.
Then there exists a constant $C(d)$ such that for all $n$ and $k$
\begin{equation}\label{eq:LCLTupper}
\sup_{x,y\in \ndtorus}P_n^k(x,y)\le C(d) \left(\frac{1}{k^{d/2}}\vee \frac{1}{n^d}\right)
\end{equation}
and in addition for any $\alpha>0$, there is a constant $C(d,\alpha)$ such that for all $n$
\begin{equation}\label{eq:LCLTlower}
\inf_{x,y\in \ndtorus,k\ge \alpha n^2} P_n^k(x,y)\ge \frac{1}{C(d,\alpha)\,n^d}
\end{equation}
and such that for all $k$
\begin{equation}\label{eq:LCLTlowerCloseGuys}
\inf_{n,x,y\in \ndtorus,\alpha (\dist(x,y))^2\le k} P_n^k(x,y)\ge \frac{1}{C(d,\alpha)\,k^{d/2}}.
\end{equation}
\end{lemma}

\begin{lemma}\label{lem:HittingTimesWithComponent}
Fix $\gamma>0$ and $d\ge 1$. Then there exists a constant 
$C_{\ref{lem:HittingTimesWithComponent}}=C_{\ref{lem:HittingTimesWithComponent}}(\gamma,d)$
so that if $S^{(n)}$ 
is any discrete time random walk on $\dntorus$ whose step distribution $\nu_n$ satisfies
\[
\nu_{n}=\gamma \nu_{\rm LSRW}+(1-\gamma) m_n
\]
where $\nu_{\rm LSRW}$ is the distribution of a step of a lazy simple random walk on $\dntorus$ and 
where $m_{n}$ is some probability measure on $\dntorus$, then the following hold.\\
(i). If $d=1$, then for all $n$,
\[
\max\{\E{\sigma_y}:y\in\onentorus\} \le C_{\ref{lem:HittingTimesWithComponent}}n^2.
\]
(ii). If $d=2$, then for all $n$,
\[
\max\{\E{\sigma_y}:y\in\twontorus\} \le C_{\ref{lem:HittingTimesWithComponent}}n^2\log n.
\]
(iii). If $d\ge 3$, then for all $n$,
\[
\max\{\E{\sigma_y}:y\in\dntorus\} \le C_{\ref{lem:HittingTimesWithComponent}}n^d.
\]
\end{lemma}

\begin{proof}
Letting $P_{S^{(n)}}^k$ be the $k$-step transition probability for $S^{(n)}$, 
we claim that there exists a constant $C(d,\gamma)$ so that for all $n$ and $k$
\begin{equation}\label{eq:LCLTupperAGAIN}
\sup_{x,y\in \ndtorus} P_{S^{(n)}}^k(x,y)\le C(d,\gamma) (\frac{1}{k^{d/2}}\vee \frac{1}{n^d})
\end{equation}
and so that for all $n$
\begin{equation}\label{eq:LCLTlowerAGAIN}
\inf_{x,y\in \ndtorus,k\ge n^2/2} P_{S^{(n)}}^k(x,y)\ge \frac{1}{C(d,\gamma)n^d}.
\end{equation}
We argue (\ref{eq:LCLTupperAGAIN}). 
Let $Y(n,k)$ be the number of times that $S^{(n)}$ uses the lazy simple random walk component up
to time $k$ and observe
that $Y(n,k)$ has a binomial distribution with parameters $k$ and $\gamma$. Therefore,
the probability of $\{Y(n,k)\le k\gamma/2\}$ is exponentially small in $k$. On the other hand,
if $Y(n,k)\ge k\gamma/2$, then
we can bound the probability of being at some $y$ at time $k$
by conditioning on the steps up to time $k$ taken by $S^{(n)}$ when it doesn't use the lazy simple 
random walk component and then using (\ref{eq:LCLTupper}) for the remaining steps. Together, this yields
\[
P_{S^{(n)}}^k(x,y)\le e^{-c(\gamma)k} + C(d) \left(\left(\frac{2}{k\gamma}\right)^{d/2}\vee \frac{1}{n^d}\right)
\]
which is bounded above by $C(d,\gamma) (\frac{1}{k^{d/2}}\vee \frac{1}{n^d})$ for some
constant $C(d,\gamma)$, obtaining (\ref{eq:LCLTupperAGAIN}).
(\ref{eq:LCLTlowerAGAIN}) is handled similarly by using (\ref{eq:LCLTlower}) instead.

Next, given a general Markov chain $X=\{X_n\}$, an element $y$ in the state space and $k\ge 1$, we
let $L_X(k,y)=L(k,y):=\sum_{i=1}^k I_{\{X_i=y\}}$ be the local time at $y$ up to time $k$. 

Case (i): $d=1$. 
By using (\ref{eq:LCLTlowerAGAIN}) and summing over $k\in \{n^2/2,\ldots, n^2\}$ and then
using (\ref{eq:LCLTupperAGAIN}) and summing over $k\in \{1,\ldots, n^2\}$, we obtain
\begin{equation}\label{eq:KeyBoundOneD}
\frac{n}{2C(1,\gamma)}\le \min_y \E{L_{S^{(n)}}(n^2,y)}\le \max_y \E{L_{S^{(n)}}(n^2,y)}\le 2C(1,\gamma)n.
\end{equation}
The  inequalities in (\ref{eq:KeyBoundOneD}) easily yield that for any $y$,
\[
\E{(L_{S^{(n)}}(n^2,y))^2}\le 8 C^2(1,\gamma)n^2\le 32 C^4(1,\gamma)(\E{(L_{S^{(n)}}(n^2,y))})^2.
\]
The Cauchy-Schwarz inequality allows us to conclude that for any $y$,
\[
\Pruu{L_{S^{(n)}}(n^2,y))>0}\ge \frac{1}{32\,C^4(1,\gamma)}.
\]
By symmetry, starting from any state, we hit $y$ after $n^2$ steps with probability at least
$\frac{1}{32\,C^4(1,\gamma)}$. This immediately gives an upper bound of the desired form.

Case (ii): $d=2$. 
Proceed exactly as in the case $d=1$ except we consider the local time up to time
$n^2\log n$ and sum (\ref{eq:LCLTlowerAGAIN}) and (\ref{eq:LCLTupperAGAIN}) 
over $k\in \{n^2/2,\ldots, n^2\log n\}$ and $k\in \{1,\ldots, n^2\log n\}$
respectively. The right and left most terms of what will become
(\ref{eq:KeyBoundOneD}) are then of order $\log n$ in this case.

Case (iii): $d\ge 3$. 
Proceed exactly as in the case $d=1$ except we consider the local time up to time
$n^d$ and sum (\ref{eq:LCLTlowerAGAIN}) and (\ref{eq:LCLTupperAGAIN}) 
over $k\in \{n^d/2,\ldots, n^d\}$ and $k\in \{1,\ldots, n^d\}$
respectively. The right and left most terms of what will become
(\ref{eq:KeyBoundOneD}) are then of order 1 in this case.
\end{proof}

\begin{lemma}\label{lem:NeededLCLTCloseGuys}
Fix $\gamma>0, R>0,$ and $d\ge 1$. Then there exists a constant 
$C_{\ref{lem:NeededLCLTCloseGuys}}=C_{\ref{lem:NeededLCLTCloseGuys}}(\gamma,R,d)$
so that if $S^{(n)}$ is any discrete time random walk on $\dntorus$  whose step distribution $\nu_n$ 
satisfies
\[
\nu_{n}=\gamma \nu_{\rm LSRW}+(1-\gamma) m_n
\]
where $\nu_{\rm LSRW}$ is the distribution of a step of a lazy simple random walk on $\dntorus$ and 
where $m_{n}$ is some probability measure on $\dntorus$
and 
\begin{equation}\label{eq:VarianceBound}
\E{\dist(S^{(n)}_1,0)^2}\le R,
\end{equation}
then it follows that for any $k$,
\begin{equation}\label{eq:LCLTlowerCloseGuysSn}
\inf_{n,x,y\in \ndtorus:\dist(x,y)^2\le k} 
P_{S^{(n)}}^k(x,y)\ge \frac{C_{\ref{lem:NeededLCLTCloseGuys}}}{k^{d/2}}.
\end{equation}
\end{lemma}

\begin{proof}
We just sketch the proof. This can be proved in a similar fashion as (\ref{eq:LCLTupperAGAIN}) but now 
using (\ref{eq:LCLTlowerCloseGuys}) as follows. By the uniform variance assumption
(\ref{eq:VarianceBound}), if we add up the 
steps taken up to time $k$ when we don't use the lazy simple random walk component, then, with fixed 
probability, this is not much larger than $\sqrt{k}$ (times a constant depending on $R$). 
Since $\dist(x,y)\le \sqrt{k}$, we won't be more than order $\sqrt{k}$ away from $y$. 
Also, with a fixed probability, the number of times we use the lazy simple random walk component up 
to time $k$ is at least $k\gamma/2$. By (\ref{eq:LCLTlowerCloseGuys}), these latter steps will bring us 
to $y$ with probability at least a constant times $\frac{1}{k^{d/2}}$.
\end{proof}

Since we are always starting in $\delta_0\times\pi_p$, we may consider all the edges to have a $\fresh$,
as in Section \ref{sec:UpperBoundSub}. We let $\tau_0=0$ and, for $j\geq 1$, let, as in (\ref{eq:Tauj's}),
$\tau_j := \min\{i > \tau_{j-1} \colon i\in \mathcal{I}\}$ (recall
(\ref{eq:DefinitionDBG}) and (\ref{eq:DefinitionI})). We have that
$\{\tau_j-\tau_{j-1}\}_{j\ge 1}$ are independent and all of these, except $\tau_1$, have the same
distribution as those introduced in (\ref{eq:Tauj's}). $\tau_1$ has a different distribution 
since we are not starting in $\regeneration$. 
As it will be simpler to look at the real time corresponding to $\tau_k$, we let 
$\tilde{\tau}_k:=\frac{C_{\ref{prop:DecreaseOfAs}.1\, \tau_k}}{\mu}$. 
Recall, as introduced in the proof of Proposition \ref{prop:DecreaseOfAs},
that $\mathcal{R}[u,v]$ denotes the range of the random walker during the time interval $[u,v]$.

\begin{lemma}\label{lem:RangeBoundedBetweenReg}
For all $d$ and for all $p \in (0,p_\critical(\mathbb{Z}^{d}))$, there are finite constants
$C_{\ref{lem:RangeBoundedBetweenReg}.1}=C_{\ref{lem:RangeBoundedBetweenReg}.1}(d,p)$
and
$C_{\ref{lem:RangeBoundedBetweenReg}.2}=C_{\ref{lem:RangeBoundedBetweenReg}.2}(d,p)$
so that for all $n$, for all $\mu$ and for all $i\ge 1$,
\[
\E{e^{C_{\ref{lem:RangeBoundedBetweenReg}.1}|\mathcal{R}[\tilde{\tau}_i,\tilde{\tau}_{i+1}]|}} 
\le C_{\ref{lem:RangeBoundedBetweenReg}.2}\, .
\]
In particular, for some constant
$C_{\ref{lem:RangeBoundedBetweenReg}.3}=C_{\ref{lem:RangeBoundedBetweenReg}.3}(d,p)$,
$\E{|\mathcal{R}[\tilde{\tau}_{i},\tilde{\tau}_{i+1}]|}\le C_{\ref{lem:RangeBoundedBetweenReg}.3}$
for all $n$, for all $\mu$ and for all $i\ge 1$.
\end{lemma}

\begin{proof}
We prove this for $i=1$. For $i\ge 2$, the distributions are of course the same while one can modify
the argument for $i=0$.
We say that a family of random variables $\{X_\alpha\}$ has {\em uniform exponential tails}
if there are constants $c_1$ and $c_2$ so that $\E{e^{c_1X_\alpha}}\le c_2$ for all $\alpha$.
The proof of Theorem \ref{thm:RenewalsHaveMeanOne}, where 
Proposition \ref{prop:GoodProcessesDie} is critically invoked, actually shows that the family of random
variables $\tau_2-\tau_1$ as we vary $n$ and $\mu$ has uniform exponential tails.
(One can similarly argue that this is also true for $\tau_1$.) Note next that
\begin{equation}\label{eq:UniformExpTails}
|\mathcal{R}[\tilde{\tau}_1,\tilde{\tau}_2]|\le 
\sum_{k=0}^{\tau_{2}-\tau_1 -1}
\left|\mathcal{R}\left[\frac{\tau_1\, C_{\ref{prop:DecreaseOfAs}.1}}{\mu}+\frac{k\,C_{\ref{prop:DecreaseOfAs}.1}}{\mu},
\frac{\tau_1\, C_{\ref{prop:DecreaseOfAs}.1}}{\mu}+\frac{(k+1)\,C_{\ref{prop:DecreaseOfAs}.1}}{\mu}\right]\right|.
\end{equation}
For $\beta=\beta(d,p)$ sufficiently small, 
we have that for all $n$ and $\mu$, the probability that, 
for $\{\eta_t\}_{t\ge 0}$ started in stationarity, a fixed edge $e$ is open at some point in 
$[0,\beta/\mu]$ is less than $p_\critical(\mathbb{Z}^{d})$. It follows from Theorem \ref{thm:exp.volume}
that started from stationarity, the family of random variables
$|\mathcal{R}[s,s+\frac{\beta}{\mu}]|$ as $n$ and $\mu$ vary has uniform exponential tails. 
Since the distribution of
$|\mathcal{R}[\frac{\tau_1\, C_{\ref{prop:DecreaseOfAs}.1}}{\mu}+s,
\frac{\tau_1\, C_{\ref{prop:DecreaseOfAs}.1}}{\mu}+s+\frac{\beta}{\mu}]|$ is just the distribution of 
$|\mathcal{R}[u,u+\frac{\beta}{\mu}]|$ conditioned on an event whose probability is uniformly bounded away 
from 0 (namely, that the edges next to the walker are closed at the appropriate time),
we have that the family of random variables 
$|\mathcal{R}[\frac{\tau_1\, C_{\ref{prop:DecreaseOfAs}.1}}{\mu}+s,
\frac{\tau_1\, C_{\ref{prop:DecreaseOfAs}.1}}{\mu}+s+\frac{\beta}{\mu}]|$ 
as we vary $n$, $\mu$ and $s$ has uniform exponential tails. Since each of the summands in
(\ref{eq:UniformExpTails}) is a sum of a uniformly bounded number of such random variables, 
the family of summands also has uniform exponential tails. Finally, this together with the fact that the 
variables $\tau_2-\tau_1$ has uniform exponential tails (as $n$ and $\mu$ vary) allows us to invoke 
Lemma \ref{lem:NacuPeres05} to conclude that the family of random variables
$\{|\mathcal{R}[\tilde{\tau}_1,\tilde{\tau}_2]|\}$ has uniform exponential tails
(as $n$ and $\mu$ vary). This is exactly the claim of the lemma.
\end{proof}

\begin{proof}[{\bf Proof of Theorem \ref{thm:hitting}}]
It is easy to show that for any fixed $n$, one can find constants so that the upper and lower bounds hold
for all $\mu$. Hence we need to only consider large $n$ below. We first establish the lower bounds. 
Fix $d$ and $p$ with $p\in (0,p_\critical(\mathbb{Z}^{d}))$. 
As usual, we choose $\beta=\beta(d,p)$ sufficiently small, so that for all $n$ and $\mu$, the probability that, 
for $\{\eta_t\}_{t\ge 0}$ started at stationarity, a fixed edge $e$ is open at some point in 
$[0,\beta/\mu]$ is less than $p_\critical(\mathbb{Z}^{d})$. 
Since $u\times\pi_p$ is stationary, it easily follows that starting from 
$\delta_0\times\pi_p$, the distribution of the cluster of the random walker at any time $t$ has
the same distribution as an ordinary cluster at the origin. It follows, as we saw in the proof of 
Lemma \ref{lem:RangeBoundedBetweenReg}, that
for any $t$, $E{|\mathcal{R}[t,t+\beta/\mu]|}\le C_{\ref{thm:hitting}.1}$ where
$C_{\ref{thm:hitting}.1}$ only depends on $d$ and $p$. 

We now move to the dimension dependent parts of the argument, going in order of increasing difficulty.

Case (iii): $d\ge 3$: lower bound. \\
It follows from the above that for any $N$,
\[
\E{|\mathcal{R}[0,N/\mu]|}\le C_{\ref{thm:hitting}.1}N/\beta.
\]
Letting $N:=\frac{\beta n^d}{3C_{\ref{thm:hitting}.1}}$, we have
\[
\E{|\mathcal{R}[0,\frac{\beta n^d}{3C_{\ref{thm:hitting}.1}\mu}]|}\le n^d/3.
\]
On the other hand, Markov's inequality implies that starting from $\delta_0\times\pi_p$, we have
that for any $n$ and $d$ and any $y\in\ndtorus$,
\[
\Pruu{\sigma_y\ge 2H_d(n)}\le 1/2
\]
from which it follows that
\[
\E{|\mathcal{R}[0,2H_d(n)]|}\ge n^d/2.
\]
We conclude that $H_d(n)\ge\frac{\beta n^d}{6C_{\ref{thm:hitting}.1}\mu}$ as desired.

\medskip\noindent
Case (i): $d=1$: lower bound. \\
Let $y$ have maximum distance from 0 in $\onentorus$. Observe that
\[
\{\sigma_y < \frac{\alpha n^2}{\mu}\} \subseteq E_1\cup E_2
\]
where
\[
E_1:=\left\{\max\left\{\dist(X_{\frac{\beta k}{\mu}},0):1\le k\le \frac{\alpha n^2}{\beta}\right\}\ge \frac{n}{8}\right\} 
\]
and
\[
E_2:=\cup_{k=1}^{\frac{\alpha n^2}{\beta}}
\left\{ |\mathcal{R}[\beta (k-1)/\mu,\beta k/\mu]|\ge n/8\right\}.
\]
By Markov's inequality, 
\[
\Pruu{E_1}\le
\E{\max\left\{\dist^2(X_{\frac{\beta k}{\mu}},0):1\le k\le \frac{\alpha n^2}{\beta}\right\}} \frac{64}{n^2}.
\]
By the proof of Theorem~\ref{thm:subMeanSquaredDisplacement} but using the maximal version of 
Lemma \ref{lem:Ball} mentioned earlier instead, we obtain that the latter term
is at most $O_{d,p}(1)\alpha$. Therefore, if $\alpha=\alpha(d,p)$ is taken sufficiently small, then
$\Pruu{E_1}\le 1/4$. Since $|\mathcal{R}[s,s+\frac{\beta}{\mu}]|$ has uniform exponential tails as $n$ and $\mu$ vary,
it follows that $\Pruu{E_2}\to 0$ as $n\to\infty$. 
We conclude that for large $n$,
$\Pruu{\sigma_y < \frac{\alpha n^2}{\mu}}\le 1/2$ and hence 
$\E{\sigma_y}\ge  \frac{\alpha n^2}{2\mu}$ yielding the desired lower bound.

\medskip\noindent
Case (ii): $d=2$: lower bound. \\
Note that we have that for all $k$,
\begin{equation}\label{eq:TaukLowerBound}
\tilde{\tau}_k\ge \frac{k\, C_{\ref{prop:DecreaseOfAs}.1}}{\mu}.
\end{equation}
By Lemma \ref{lem:RangeBoundedBetweenReg}, there exists $R$ such that the second moments of 
$|\mathcal{R}[\tilde{\tau}_i,\tilde{\tau}_{i+1}]|$ are bounded uniformly (in $n$ and $\mu$) by $R$.
We also let $\gamma$ be such that the distribution of
$X_{\tilde{\tau}_{i+1}}-X_{\tilde{\tau}_i}$ contains, for all $n$ and $\mu$, $\gamma \nu_{\rm LSRW}$.

For simplicity, let 
$C_{\ref{thm:hitting}.2}:=\frac{C_{\ref{lem:NeededLCLTCloseGuys}}(\gamma,R,d)}
{16\,C_{\ref{lem:RangeBoundedBetweenReg}.3}}$.
If there exists $y\in\dntorus$ such that
\begin{equation}\label{eq:AssumptionForSomeY}
\Pruu{\sigma_y\le \tilde{\tau}_{C_{\ref{thm:hitting}.2}n^2\log n}}\le 3/4,
\end{equation}
then, by (\ref{eq:TaukLowerBound}), we have
\[
\Pruu{\sigma_y\le \frac{C_{\ref{prop:DecreaseOfAs}.1}\,C_{\ref{thm:hitting}.2}\,n^2\log n}
{\mu}}\le 3/4.
\]
This would lead to
\[
\E{\sigma_y}\ge \frac{C_{\ref{prop:DecreaseOfAs}.1}\,C_{\ref{thm:hitting}.2}\,n^2\log n}
{4\,\mu}
\]
and we would be done. Hence we may assume that for all $y\in\ndtorus$
\begin{equation}\label{eq:Assumption}
\Pruu{\sigma_y\le \tilde{\tau}_{C_{\ref{thm:hitting}.2}\,n^2\log n}}\ge 3/4.
\end{equation}
We will prove below that under this assumption, for all large $n$, we have that
\begin{equation}\label{eq:NeededBound}
\E{\left|\mathcal{R}[0,\tilde{\tau}_{C_{\ref{thm:hitting}.2}\,n^2\log n}]\right|}\le \frac{n^2}{4}.
\end{equation}
On the other hand, by summing (\ref{eq:Assumption}) over $y$, one immediately obtains
\begin{equation}
\E{\left|\mathcal{R}[0,\tilde{\tau}_{C_{\ref{thm:hitting}.2}n^2\log n}]\right|}\ge \frac{3n^2}{4}.
\end{equation}
This contradiction tells us that in fact 
(\ref{eq:Assumption}) cannot be true for all $y$ and so we would have our desired lower bound on
the hitting time.

We are now left to prove that (\ref{eq:Assumption}) implies (\ref{eq:NeededBound}). 
First we abbreviate the event
$\{\sigma_y\le \tilde{\tau}_{C_{\ref{thm:hitting}.2}\,n^2\log n}\}$ 
by $U_y$. Next, for $i\ge 1$, let $Y_i:=\mathcal{R}[\tilde{\tau}_{i-1},\tilde{\tau}_i]$.
Clearly for all $y\in\dntorus$, we have that
\[
\Pruu{U_y}\le
\frac{
\E{\sum_{i=1}^{2C_{\ref{thm:hitting}.2}\,n^2\log n}
I_{\{y\in Y_i\}}}}
{\E{\sum_{i=1}^{2C_{\ref{thm:hitting}.2}\,n^2\log n}
I_{\{y\in Y_i\}}\mid U_y}}\, .
\]
We will argue further down that for all $y$
\begin{equation}\label{eq:NeededLogBound}
\E{\sum_{i=1}^{2C_{\ref{thm:hitting}.2}\,n^2\log n}
I_{\{y\in Y_i\}}\mid U_y} \ge C_{\ref{lem:NeededLCLTCloseGuys}}(\gamma,R,d) \log n/2.
\end{equation}
Assuming this for the moment, we plug this into the previous inequality, sum over $y\in\ndtorus$, use
Lemma \ref{lem:RangeBoundedBetweenReg} and conclude (\ref{eq:NeededBound}). This leaves us with 
proving (\ref{eq:NeededLogBound}).

Let $G_i:=\{|\mathcal{R}[\tilde{\tau}_{i-1},\tilde{\tau}_i]|\le (\log n)^2\}$ and 
$G=\cap_{i=1}^{n^3}G_i$. Using Lemma \ref{lem:RangeBoundedBetweenReg}, one easily sees that
\begin{equation}\label{eq:GhasProbOne}
\lim_{n\to\infty} \Pruu{G}=1.
\end{equation}
Note that (\ref{eq:GhasProbOne}) and (\ref{eq:Assumption}) imply that
\[
\Pruu{G\mid U_y}=1-o(1).
\]
Let $T:=\min\{i:y\in Y_i\}$. By writing $\Pruu{G\mid U_y}$ as
\[
\sum_{j=1}^{C_{\ref{thm:hitting}.2}\,n^2\log n}
\Pruu{G\mid T=j}\Pruu{T=j\mid U_y},
\]
we see that there must be a set 
$J_n\subseteq \{1,2,\ldots,C_{\ref{thm:hitting}.2}\,n^2\log n\}$ 
so that 
\[
\Pruu{T\in J_n\mid U_y}=1-o(1) \mbox{ and }\inf_{j\in J_n}\Pruu{G\mid T=j}=1-o(1).
\]
These imply that
\[
\E{\sum_{i=1}^{2C_{\ref{thm:hitting}.2}\,n^2\log n}
I_{\{y\in Y_i\}}\mid U_y}=
\sum_j\E{\sum_{i=1}^{2C_{\ref{thm:hitting}.2}\,n^2\log n}
I_{\{y\in Y_i\}}\mid T=j}\Pruu{T=j\mid U_y}\ge
\]
\[
(1-o(1))\min_{j\in J_n}\E{\sum_{i=1}^{2C_{\ref{thm:hitting}.2}\,n^2\log n}
I_{\{y\in Y_i\}}\mid T=j}\ge
\]
\[
(1-o(1))\min_{j\in J_n}\Pruu{G_j\mid T=j}
\E{\sum_{i=1}^{2C_{\ref{thm:hitting}.2}\,n^2\log n}
I_{\{y\in Y_i\}}\mid G_j,T=j}\ge
\]
\[
(1-o(1))\min_{j\in J_n}
\E{\sum_{i=1}^{2C_{\ref{thm:hitting}.2}\,n^2\log n}
I_{\{y\in Y_i\}}\mid G_j,T=j}\ge
\]
\[
(1-o(1))\min\left\{
\sum_{i=j+(\log n)^4}^{j+n}\Pruu{X_{\tilde{\tau}_i}=y\mid G_j,T=j}:
j\in \{1,2,\ldots,C_{\ref{thm:hitting}.2}\,n^2\log n\}\right\}.
\]
Now, recall $R$ was chosen so that the i.i.d.\ increments of
$X_{\tilde{\tau}_j}, X_{\tilde{\tau}_{j+1}},\ldots$ have variances bounded by $R$,
uniformly in $n$ and $\mu$. In addition, these increments are independent of the event $\{G_j,T=j\}$. The occurence of
$G_j$ implies that $\dist(y,X_{\tilde{\tau}_j})\le (\log n)^2$. It then follows from Lemma \ref{lem:NeededLCLTCloseGuys} 
that for $i\in \{j+(\log n)^4,\ldots, j+n\}$, 
$\Pruu{X_{\tilde{\tau}_i}=y\mid G_j,T=j}\ge \frac{C_{\ref{lem:NeededLCLTCloseGuys}}(\gamma,R,d)}{i-j}$.
Summing over our set of $i$ yields (\ref{eq:NeededLogBound}).

\medskip\noindent
We now move to the upper bounds which can be obtained simultaneously for all dimensions with 
an application of Lemma \ref{lem:HittingTimesWithComponent} as follows. 

Letting $U_k:=X_{\tilde{\tau_k}}$ as in (\ref{eq:Tauj's}), we have that
$\{U_k-U_{k-1}\}_{k\ge 1}$ are independent with all, except the first, having the same distribution. 
(Recall $d$ and $n$ are suppressed in the notation.)
We will apply Lemma \ref{lem:HittingTimesWithComponent} to the random walk on $\dntorus$ 
whose steps have distribution $\{U_2-U_{1}\}$. The key hypothesis of this lemma holds as indicated
in (\ref{eq:SRWcomponent}). Letting $\hat{\sigma}_y$ denote the hitting time for our induced 
discrete time system, Lemma \ref{lem:HittingTimesWithComponent} yields that $\E{\hat{\sigma}_y}$ 
is bounded above by $O_{d,p}(1)n^2$, $O_{d,p}(1)n^2\log n$ and $O_{d,p}(1)n^d$ 
respectively in dimensions 1,2 and $\ge 3$. 
(Note that while the constant in Lemma \ref{lem:HittingTimesWithComponent} depends on $\gamma$,
our $\gamma$ here only depends on $d$ and $p$.)

Note next that
\[
\sigma_y\le \sum_{i=1}^{\hat{\sigma}_y} \tilde{\tau}_i-\tilde{\tau}_{i-1}
\]
By Wald's lemma, we obtain
\[
\E{\sigma_y}\le \frac{C_{\ref{prop:DecreaseOfAs}.1}}{\mu}\E{\tau_2-\tau_1}\E{\hat{\sigma}_y}.
\]

We cheated a little since $\tau_1$ has a different distribution than the other $\tau_i-\tau_{i-1}$'s;
however, as mentioned earlier, the proof of Theorem \ref{thm:RenewalsHaveMeanOne} easily yields that
$\E{\tau_1}$ is also bounded by a constant depending only on $d$ and $p$ and so this can be ignored in
the above. Finally, Theorem \ref{thm:RenewalsHaveMeanOne} and our derived bound on 
$\E{\hat{\sigma}_y}$ yields the desired upper bound on $\E{\sigma_y}$.
\end{proof}

\section{Open questions}
\label{sec:question}

In Theorem \ref{thm:CLT}, the diffusion constant $\sigma^2$ depends on $d,p$ and $\mu$.
It is easy to see that $\lim_{\mu \to\infty} \sigma^2(d,p,\mu)=p$ since in the large
$\mu$ limit, this just corresponds to random walk time scaled by $p$. Also
Corollary \ref{cor:subMeanSquaredDisplacement} implies that for $p<p_c(d)$, we have that
$\liminf_{\mu \to 0} \frac{\sigma^2(d,p,\mu)}{\mu}<\infty$. One can also show that the 
corresponding $\limsup$ is positive in this case as well.

1. Show that $\sigma^2(d,p,\mu)$ is increasing in both $p$ and $\mu$. 

2. Show that even $\frac{\sigma^2(d,p,\mu)}{\mu}$ is increasing in $\mu$.

\section*{Acknowledgements}
A major part of this work was carried out when the second and third authors were
visiting (at different times) Microsoft Research at Redmond, WA, and they thank
the Theory Group for its hospitality and for creating a stimulating research environment.
JES also acknowledges the support of the Swedish Research Council and the Knut and Alice Wallenberg
Foundation.


\bibliographystyle{plain}
\bibliography{rw}

\noindent Yuval Peres\\
Microsoft Research\\
One Microsoft Way\\
Redmond, WA\\
98052-6399\\
USA\\
{peres@microsoft.com} \\
{\tt {http://research.microsoft.com/en-us/um/people/peres/}}

\medskip\noindent Alexandre Stauffer\\ 
Department of Mathematical Sciences  \\
University of Bath  \\
Claverton Down, Bath  \\
BA2 7AY, U.K.  \\
{a.stauffer@bath.ac.uk}\\
{\tt {http://people.bath.ac.uk/ados20/}}

\medskip\noindent Jeffrey E. Steif\\
Mathematical Sciences \\
Chalmers University of Technology \\
and \\
Mathematical Sciences \\
G\"{o}teborg University \\
SE-41296 Gothenburg, Sweden \\
{steif@math.chalmers.se} \\
{\tt {http://www.math.chalmers.se/\string~steif/}}

\end{document}